\documentclass[12pt]{article}

\usepackage{amsmath}
\usepackage{amsfonts}
\usepackage{amssymb}
\usepackage{amscd}
\usepackage{amsthm}
\usepackage{textcomp}
\usepackage{bbm}
\usepackage{color}
\usepackage[linktocpage=true,plainpages=false,pdfpagelabels=false]{hyperref}

\input xypic

\usepackage[matrix,arrow,curve]{xy}

\usepackage{color}

\newcommand{\quash}[1]{}


\newtheorem{defin}{Definition}
\newtheorem{prop}{Proposition}
\newtheorem{nt}{Remark}
\newtheorem{Th}{Theorem}
\newtheorem{lemma}{Lemma}
\newtheorem{cons}{Corollary}

\newtheorem{example}{Example}
\newtheorem{defin-prop}{Definition-proposition}

\newfont{\ssdbl}{msbm8}
\newfont{\sdbl}{msbm9}
\newfont{\dbl}{msbm10 at 12pt}

\newcommand{\oo}{{\cal O}}

\newcommand{\res}{\mathop {\rm res}}
\newcommand{\Hom}{\mathop {\rm Hom}}

\newcommand{\Frac}{\mathop {\rm Frac}}

\newcommand{\dz}{\mathbb{Z}}
\newcommand{\dc}{\mathbb{C}}
\newcommand{\dq}{\mathbb{Q}}

\newcommand{\Ker}{{\rm Ker}\:}
\newcommand{\Image}{{\rm Im}\:}

\newcommand{\lrto}{\longrightarrow}

\newcommand{\dn}{\mathbb{N}}

\newcommand{\df}{\mathbb{F}}

\newcommand{\F}{{\bf F}}

\newcommand{\D}{{\cal D}}

\newcommand{\ord}{{\mathop{\rm ord}}}

\begin{document}

\author{
D. V. Osipov, A. N. Parshin
}

\title{Harmonic analysis on the rank-$2$ value group of a two-dimensional local field \thanks{The first author is partially supported by Laboratory of Mirror Symmetry NRU HSE, RF Government grant, ag.~\textnumero 14.641.31.0001.}}
\date{}

\maketitle

\begin{abstract}
In this work we construct  harmonic analysis on free Abelian groups of rank  $2$, namely:
we construct and investigate spaces of functions and distributions,
Fourier transforms, actions of discrete and extended discrete Heisenberg groups.
In case of the rank-$2$ value group of a two-dimensional local field with  finite last residue field we connect this harmonic analysis
with harmonic analysis on the two-dimensional local field, where the latter harmonic analysis was  constructed  in earlier works by the  authors.
\end{abstract}

\section{Introduction}

In the theory of zeta functions of algebraic varieties and schemes, one of the main problems is the analytic continuation and the functional equations of zeta and $L$-functions. For one-dimensional schemes, there is the  Tate-Iwasawa method that goes back to Riemann and Hecke. The method applies harmonic analysis on the adele groups of such schemes to a  solution of  these problems.

 The report~\cite{Par6} raised the question (see also work~\cite{Par5}) about the extension of the Tate-Iwasawa method to the case of two-dimensional schemes  of finite  type over $\mathop{\rm Spec} \dz$  using the theory of higher adeles and harmonic analysis constructed in works~\cite{OsipPar1,OsipPar1'}.

In work~\cite{Par4} a new version of the method was  introduced. This new version follows the lines of paper~\cite{Par3}, where  the well-known manipulations with formulas were removed and replaced  by functoriality and duality considerations.
The construction  was applied to a  new proof of  the functional equation for the $L$-functions. In the talks~\footnote{{\em Zeta-functions of schemes
in dimensions $1$ and $2$},
Steklov Mathematical Institute of RAS, seminar of departments of algebra and algebraic geometry (I.~R.~Shafarevich seminar), 7th and 14th  November 2017; {\em Fourier analysis and Poisson formulas on discrete groups and analytic
properties of zeta-functions of algebraic varieties}, Arithmetic and Analysis, A conference on the occasion of Christopher Deninger's 60th birthday, April 9 -  13, 2018, M\"unster University, Germany.} of 2017, 2018, the second author discussed the possible extension of this version of  the Tate-Iwasawa  method to the case of two-dimensional schemes, first of all to the algebraic surfaces defined over a finite field.

For the case of curves, the new approach is based on the transition from locally compact groups of adeles and ideles to discrete groups of divisors and classes of divisors on curves and their dual complex tori.
 This transition was made in paper~\cite{Par3}. In the above mentioned talks the case of algebraic surfaces defined over a finite field was considered and
  it was shown that the basic analytical problems of zeta and $L$-functions of algebraic surfaces can be reformulated as the task of construction of a  harmonic analysis (the Fourier transform and the Poisson formula) on discrete groups of  $0$-cycles and classes of $0$-cycles on an algebraic surface. In case of an algebraic surface it is these groups that are connected with the zeta-function of the surface just as in case of a curve its zeta-function is connected with groups of divisors (which are $0$-cycles in this case).
  The principal difference between a  surface and a curve is that in case of surface there are cycles of
two types: dimension $0$ ($0$-cycles) and dimension $1$ (divisors).

   Earlier in papers of authors~\cite{OsipPar1} and \cite{OsipPar1'}, harmonic analysis was constructed on adele groups of an algebraic surface
   (function and distribution spaces, Fourier transforms, Poisson formulas). These groups appear from adele ring $\mathbb A$ of the algebraic surface. With these groups one connects subgroups of the multiplicative group  $\mathbb A^*$ of invertible elements.
   In paper~\cite{Par6} the discrete adele groups, in particular various groups of divisors,  were constructed in this situation. The question on construction of harmonic analysis on these groups (such that this analysis is based on
   already constructed harmonic analysis from papers~\cite{OsipPar1} and~\cite{OsipPar1'}) by analogy with the case of curves in~\cite{Par3}
   was not considered at that time. After the appearance of paper~\cite{Par4} and the  above mentioned reformulation of properties of the zeta-function of a surface in the  form of  analysis on the group of $0$-cycles the following program appears.

First, one needs to transfer harmonic analysis of papers~\cite{OsipPar1} and \cite{OsipPar1'}
to discrete groups of divisors and divisor classes. Then, a transition from the groups associated with divisors
to the groups of $0$-cycles, which are directly related to the study of zeta- and $L$-functions, is necessary.  Thus, the entire program on an algebraic surface
is divided into the following two parts:

\begin{description}
\item[1)]	\label{p1}
the construction of harmonic analysis on adele discrete groups, in particular on groups of divisors and their classes;
\item[2)]	
establishing the connection of such harmonic analysis with harmonic analysis on adele discrete groups associated with groups of $0$-cycles
and their classes.
\end{description}

The second part is much more complicated, but, nevertheless, it can be assumed that the main tool for establishing such a connection should be
adele Heisenberg groups on an algebraic surface, which were also introduced in~\cite{Par6}.
In this paper we consider only the first part and only the local discrete groups associated with fixed two-dimensional local field of
an algebraic surface.  (However, most of the work also applies to an arbitrary two-dimensional local field with  finite last residue field.)

The fact is that  the adele discrete groups which we are interested in and which are associated with a whole algebraic surface
are products (direct sums, adele products) of local discrete groups  associated with two-dimensional local fields.
So, in this paper we  implement  the first step of part~$\bf 1)$ of this program, consisting in considering the local case, i.e. for two-dimensional local fields $K$ with finite last residue field and associated  free Abelian groups $ \Gamma_K$ of rank $2$, which are the rank $2 $-value groups of these fields, and noncommutative discrete Heisenberg groups.
We note that the complexity of the problem increases dramatically when one moves to the two-dimensional case. In the case of dimension $1$, the constructions whose  generalization is the object of this work   occupies $3$ pages in paper~\cite{Par3}.
Moreover, in~\cite{Par3} there is also a transition to the complex torus, which is dual to the rank-$1$ value group  (this is the group $ \mathbb Z $).

In the case of a rank-$2$ group, such a transition is also possible and this will be done in a separate paper.

In this paper we construct harmonic analysis on function and distribution spaces on a free Abelian group  $\Gamma$ of rank $2$ such that
the following central extension is fixed:
$$
0 \lrto \dz \lrto \Gamma \stackrel{\pi}{\lrto} \dz \lrto 0 \mbox{.}
$$
(If $\Gamma$ is $\Gamma_K$
which is the rank-$2$ value group for a two-dimensional local field $K$, then
$\pi$
is a homomorphism given by the discrete valuation 
 on the field
 $K$.)

For arbitrary fixed (nonzero complex) number $r \in \dc^*$
we define infinite-dimensional $\dc$-vector function and distribution spaces on the group $\Gamma$:
${\mathcal{D}}_{+, \alpha, r}(\Gamma)$
and ${\mathcal{D}}'_{+, \alpha, r}(\Gamma)$,
where the index $\alpha \in \Gamma \cup \pi(\Gamma)$.
These spaces are subspaces or quotient spaces of direct products or direct sums
of tensor products of certain function and distribution spaces on $\dz$-torsors $\pi^{-1}(k)$, where $k \in \dz$,
with canonical one-dimensional $\dc$-vector spaces which depend on $k$ and $\alpha$.
Moreover, the function and distribution spaces on a $\dz$-torsor $\pi^{-1}(k)$
coincide with the spaces of coinvariants and invariants of the function and distribution spaces on a one-dimensional vector space over a one-dimensional local field when the $\dz$-torsor appears from this one-dimensional space after the action of the group of invertible elements of the discrete valuation ring of this field.
We note that linear conditions which define subspaces or quotient spaces   ${\mathcal{D}}_{+, \alpha, r}(\Gamma)$
and ${\mathcal{D}}'_{+, \alpha, r}(\Gamma)$ from infinite direct products or direct sums
depend also on a number  $r \in \dc^*$.

We define and investigate
a natural  $\dc$-linear pairing between the spaces   ${\mathcal{D}}_{+, \alpha, r}(\Gamma)$
and ${\mathcal{D}}'_{+, \alpha, r}(\Gamma)$. We define also Fourier transforms on these spaces.
If the group $\Gamma$ is the group  $\Gamma_K$,
and the last residue field of a two-dimensional local field  $K$ is a finite field  $\df_q$,
then for  $r = q$  the constructed function and distribution spaces  and Fourier transforms on them
are connected with the function and distribution spaces on  $K$ and  Fourier transforms there
(they were constructed and investigated in~\cite{OsipPar1} and~\cite{OsipPar1'}).

Besides, the groups which are isomorphic to the discrete Heisenberg group  ${\rm Heis}(3, \dz)$
and the  extended discrete Heisenberg group ${\rm Heis}(3, \dz) \rtimes \dz$ naturally act on the spaces
${\mathcal{D}}_{+, \alpha, r}(\Gamma)$
and  ${\mathcal{D}}'_{+, \alpha, r}(\Gamma)$.

The work consists of $6$ sections. Let us proceed to the exposition  of its content.

Section~\ref{se-2} of the paper gives a description of general properties of one- and two-dimensional local fields: their examples,  valuation rings and  value groups, structure of linear order on the rank-$2$ value group, disjoint decompositions of these fields.

Reminders about the function  and distribution spaces on two-dimensional local fields (following papers~\cite{OsipPar1} and~\cite{OsipPar1'}), as well as new properties of these spaces, are given in section~\ref{Funct_distr}.
Here we describe the $\dc$-vector function and distribution spaces
$\D_{\alpha}(K)$ and $\D'_{\alpha}(K)$
 on a two-dimensional local field $K$ with a partial topology (i.e, on a two-dimensional local field with  finite last residue field, and this local field is endowed with certain additional data, see section~\ref{part-topol}.)
 Here an index $\alpha$
 belongs to the set $\Gamma_K \cup \pi(\Gamma_K)$.
 Besides, in section~\ref{charcteristic}
 the characteristic functions for  fractional ideals of the field  $K$---delta-functions---are described.
 Further, in sections~\ref{inv-sec}-\ref{inv-dual}
 we give a detailed description of the spaces of invariants and coinvariants with respect to the action of the group of invertible elements
 $(\oo'_K)^*$
  of the rank-$2$ valuation ring of the field $K$
  and we prove  that the natural pairing between these spaces is non-degenerate.

The main part of the paper is section~\ref{dis},  where we define  function and distribution   spaces
${\mathcal{D}}_{+, \alpha, r}(\Gamma)$
and ${\mathcal{D}}'_{+, \alpha, r}(\Gamma)$
on a group $\Gamma$ (with a fixed homomorphism $\pi$).
We describe the connection of these spaces with the spaces of coinvariants and invariants
of the function and distribution spaces on a two-dimensional local field $K$ with last residue field $\df_q$
(and with certain additional data from section~\ref{part-topol}) when $\Gamma = \Gamma_K$ and $r = q$.

In section~\ref{Sft} we construct and investigate properties of the Fourier transforms on the spaces
${\mathcal{D}}_{+, \alpha, r}(\Gamma)$
and ${\mathcal{D}}'_{+, \alpha, r}(\Gamma)$.  In case  $\Gamma = \Gamma_K$ and $r =q$, we describe the connection of these Fourier transforms with
the Fourier transforms from harmonic analysis constructed in paper~\cite{OsipPar1} (see section $5$ there).

In section~\ref{Hei} we construct and investigate an action of the discrete Heisenberg group
${\rm Heis}(3, \dz)$
and extended discrete Heisenberg group  ${\rm Heis}(3, \dz) \rtimes \dz$
on the spaces  ${\mathcal{D}}_{+, \alpha, r}(\Gamma)$
and ${\mathcal{D}}'_{+, \alpha, r}(\Gamma)$.

In conclusion, we would like to note several works written quite a long time ago and not directly related to the theory of higher local fields and adeles, but nevertheless these works  had a great influence on our research.
This is a book by A.~Pressley and G.~Segal~\cite{PS},  in which constructions of loop groups on smooth manifolds are presented.
An analogy between loop groups and two-dimensional local fields was noted in~\cite{Par6}
and the existing ``loop rotation'' operation was transferred to the theory of two-dimensional local fields and served as the starting point for the definition of the extended discrete Heisenberg group.
Another important work is a paper by E.~Arbarello,  C.~De~Concini, V.~G.~Kac~\cite{ADK}, which led to the categorical constructions of extensions and symbols in the works of the first author and these constructions are partially used in this work.

\section{Two-dimensional local fields} \label{se-2}

\subsection{Definition and examples}

We recall the following definition.

\begin{defin}
A field $L$ is a one-dimensional local field (or, simply, a local field) with  residue field $k$ if $L$ is a complete  discrete valuation field with  residue field $k$.

A field $K$ is a two-dimensional local field with  last residue field $k$ if $K$ is a complete  discrete valuation field with  residue field $\bar{K}$ such that $\bar{K}$ is a one-dimensional local field with residue field $k$.
\end{defin}

Typical examples of two-dimensional local fields when $k$ is a finite field are the following:
\begin{equation}  \label{examples}
\df_q ((u))((t))  \, \mbox{,} \qquad \dq_p((t)) \, \mbox{,}  \qquad \dq_p\{\{ u \}\}  \, \mbox{,}
\end{equation}
where the last field is the completion of the field $\Frac \dz_p[[u]]$ with respect to the discrete valuation given by the height-$1$  principal prime ideal $(p)$. Up to finite extensions of two-dimensional local fields, this list of examples is full. Besides, two-dimensional local fields appear naturally from flags of subschemes on algebraic or ariphmetic surfaces, see more, e.g., in~\cite{Osip1}.

Let $K$ be a two-dimensional local field with  last residue field $k$. By $\oo_K$ we denote the discrete valuation ring associated with the discrete  valuation of $K$.
Let $\oo_{\bar{K}}$ be the discrete valuation ring associated with the discrete valuation of $\bar{K}$. Then we have
$$
K \subset \oo_K  \stackrel{p_1}{\lrto} \bar{K}  \supset \oo_{\bar{K}}  \stackrel{p_2}{\lrto} k  \, \mbox{,}
$$
where $p_1$ and $p_2$ are the natural maps.

We denote a ring:
 $$\oo_{K}' = p_1^{-1} (\oo_{\bar{K}}) \, \mbox{.}$$

For example, if $K = k((u))((t))$, then
$$
\oo_K = k((u))[[t]]  \, \mbox{,} \qquad \bar{K} = k((u))  \, \mbox{,}  \qquad \oo_{\bar{K}}= k[[u]]  \, \mbox{,} \qquad \oo_K' =
k[[u]] + tk((u))[[t]]  \, \mbox{.}
$$

\smallskip

Let $t$ be a local parameter in the discrete valuation ring $\oo_K$. We denote $\oo_K$-modules
$$
\oo(n) = t^n \oo_K  \subset K \, \mbox{,} \qquad  \mbox{where} \qquad n \in \dz \, \mbox{.}
$$
We note that $\oo(n)$ does not depend on the choice of $t$, and the set of $\oo_K$-modules  $\oo(n)$ for all $n \in \dz$ is the set of all $\oo_K$-submodules of $K$.

\subsection{The value group $\Gamma_K$}

Let $K$ be a two-dimensional filed with  last residue field $k$.

\begin{defin}
We define a group
$$
\Gamma_K = K^* / (\oo_K')^*  \, \mbox{.}
$$
\end{defin}

There is a canonical exact sequence
\begin{equation}  \label{sp0}
0 \lrto \dz  \lrto \Gamma_K \stackrel{\pi}{\lrto} \dz  \lrto 0  \, \mbox{,}
\end{equation}
where the map $\pi$ is induced by the discrete valuation  $K^*  \lrto \dz$, and the group $\Ker \pi$ is the value group of the discrete valuation of  $\bar{K}$,
which is canonically isomorphic to $\dz$.

Fixing of a local parameter $t$ of $\oo_K$ provides a splitting of  exact sequence~\eqref{sp0}. This gives an isomorphism:
\begin{equation}  \label{spl}
\Gamma_K  \simeq \dz \oplus \dz  \, \mbox{,}
\end{equation}
where the first group $\dz$ corresponds to $\pi(\Gamma_K)$.

A pair of elements $u, t$ from $\oo_{K}'$ is called local parameters of a two-dimensional local field $K$
if $t$ is a local parameter of the discrete valuation ring $\oo_K$ and $p_1(u)$ is a local parameter of the discrete valuation ring $\oo_{\bar{K}}$.

If we choose another local parameter $t'$ of $\oo_K$, then $t' = t u^k a$, where $a \in (\oo_K')^*$, and  the isomorphism
in formula~\eqref{spl} will change by the action of a matrix:
$$\left(
\begin{matrix}
1 & k \\
0 & 1
\end{matrix} \right) \, \mbox{.}
$$
Thus, we have an explicit isomorphism
\begin{equation}  \label{ch}
 \dz \oplus \dz \ni { (n \oplus p)  \longmapsto (n  \oplus n +pk)}  \in \dz \oplus \dz \, \mbox{,}
\end{equation}
which gives also  a homomorphism $\dz  \lrto  {\rm GL}(2, \dz)= {\rm Aut \,} (\dz \oplus \dz)$, where $k \in \dz$.

The group $\Gamma_K$ is naturally linearly ordered. Indeed, for $\gamma_1, \gamma_2$ from $\Gamma_K$ we put
$$
\gamma_1 > \gamma_2   \qquad \mbox{iff}  \qquad
\pi(\gamma_1) > \pi(\gamma_2)  \qquad \mbox{or} \qquad
 \pi(\gamma_1) = \pi(\gamma_2) \, \mbox{,}
\quad   \Ker \pi \ni \gamma_1 - \gamma_2 >0  \, \mbox{.}
$$
This order goes to lexicographical order after the fixing of splitting  of exact sequence~\eqref{sp0} and consequently of isomorphism~\eqref{spl}.

Thus, we have the rank-$2$ valuation
$$ \nu_K \; :   \; K^* \lrto \Gamma_K \mbox{.}$$

For any $\gamma \in \Gamma_K$ we denote $\oo_{K}'$-submodule of $K$:
$$
\oo'(\gamma) = (0) \sqcup \coprod_{\gamma' \ge \gamma} \nu_K^{-1}(\gamma')  = a \oo'_K \, \mbox{,}
$$
where $(0)$ is the zero element of the group $K$, and $a \in K^*$ such that $\nu_K(a)= \gamma$.

From~\cite[Theorem~1]{Par1} we know that any $\oo_K'$-submodule of $K$ is either $\oo(n)$ for appropriate $n \in \dz$
or $\oo'(\gamma)$ for appropriate $\gamma \in \Gamma_K$. Moreover,  $\oo_K'$-submodules $\oo(n)$ and  they alone are infinitely generated $\oo_K'$-submodules.

\medskip

To describe all $\oo_K'$-submodules of $K$ at once, we introduce a set
$$\breve{\Gamma}_K = \Gamma_K \cup  \pi(\Gamma_K)  \, \mbox{.}$$

We will write an element from $\Gamma_K \subset \breve{\Gamma}_K$ as a pair $(n,p)$, where $n \in \pi(\Gamma_K) = \dz$, ${p \in \pi^{-1}(n)}$,
and an element from $\pi(\Gamma_K) = \breve{\Gamma}_K \setminus \Gamma_K$ as a pair $(n , - \infty)$, where $n \in \dz$.
We have a natural map
$$\breve{\pi} \; :  \;  \breve{\Gamma}_K  \lrto \dz $$
which extends the map $\pi$ and it is the identity map on $ \pi(\Gamma_K) $.

Then the order on $\Gamma_K$ is extended to a linear order on $\breve{\Gamma}_K$
by additional rules:
\begin{eqnarray*}
\mbox{for}  \quad \alpha_1, \alpha_2 \in \breve{\Gamma}_K \quad  \mbox{we have} \quad \alpha_1  <   \alpha_2   \quad \mbox{if} \quad \breve{\pi}(\alpha_1) < \breve{\pi}(\alpha_2) \, \mbox{;}\\
(n, -\infty)  < (n, p) \, \mbox{,} \quad \mbox{where} \quad p \in \pi^{-1}(n) \, \mbox{.}
\end{eqnarray*}

We note that the group $\Gamma_K$ acts on the set  $\breve{\Gamma}_K$ such that  this action extends the group operation in $\Gamma_K$, and the formula for the remaining action of $\Gamma_K$ to $\breve{\Gamma}_K \setminus \Gamma_K$ is:
\begin{eqnarray*}
(n, p) \circ (m, -\infty) = (n+m, -\infty) \, \mbox{.}
\end{eqnarray*}

\smallskip

Using that $\pi(\Gamma_K)=\dz$ we canonically identify now the set of all $\oo_{K}'$-submodules in $K$ with the set $\breve{\Gamma}_K$.
For any $\alpha \in \breve{\Gamma}_K$ we denote by $\oo'(\alpha) \subset K$ the corresponding $\oo_K'$-submodule (if $\alpha = (n, -\infty) \in \pi(\Gamma_K)$, then $\oo'(\alpha)= \oo(n)$).
We note that for any $\alpha_1, \alpha_2 \in \breve{\Gamma}_K$ we have
$$
\alpha_1 \le \alpha_2  \qquad \mbox{iff} \qquad \oo'(\alpha_1) \supset \oo'(\alpha_2)  \, \mbox{.}
$$

\subsection{Canonical decompositions}  \label{can-dec}

Let $K$ be a two-dimensional field with  last residue field $k$.

Let $m_{\bar{K}}$ be the maximal ideal of $\oo_{\bar{K}}$. Then $m_K'= p_1^{-1} (m_{\bar{K}})$
is the maximal ideal of~$\oo_{K}'$.

We introduce the following decompositions for $K$ and $\bar{K}$.

For any $n \in \dz$ and any $\gamma \in \Gamma_K$ we denote
$$\oo_{\bar{K}}^*(n) = m_{\bar{K}}^n \setminus m_{\bar{K}}^{n+1} \subset \bar{K}^* \, \mbox{, } \qquad \qquad \oo^*(\gamma)= \oo'(\gamma) \setminus
m_{K}' \oo'(\gamma) \subset K^* \, \mbox{.}
$$
Then we have disjoint unions of sets:
$$
\bar{K} =(0)  \sqcup  \coprod_{n \in \dz} \oo^*_{\bar{K}}(n)  \, \mbox{,}  \qquad \qquad
K = (0)  \sqcup \coprod_{\gamma \in \Gamma_K} \oo^*(\gamma)  \, \mbox{,}
$$
where $(0)$ is the identity element of the group $\bar{K}$ or the group $K$.

\section{Function and distribution spaces on $K$}  \label{Funct_distr}
\subsection{Two-dimensional local fields with partial topologies}   \label{part-topol}
Let $K$ be a two-dimensional local field with   finite last residue field $\df_q$.

The additive group of the  field $\bar{K}$ is a locally compact Abelian group. Therefore for any $n \in \dz$ the one-dimensional  $\bar{K}$-vector space
$\oo(n)/ \oo(n+1)$ is a totally disconnected locally compact Abelian group (or, equivalently, a locally profinite Abelian group).

\begin{defin}
We will say that
a two-dimensional local field $K$ with  finite last residue field is endowed with  a partial topology if the following conditions are satisfied.
\begin{enumerate}
\item For any $n \le m \in \dz$ the  Abelian group $\oo(n)/\oo(m)$ is endowed with  a topology such that this group is a totally disconnected locally compact group, and the topology on the group $\oo(n)/ \oo(n+1)$ (when $m =n+1$) comes from the topology on~$\bar{K}$.
\item For any $n  \le m \le k \in \dz$ the topology on the subgroup $\oo(m) / \oo(k)$ is induced from the topology on the group $\oo(n)/\oo(k)$, and the topology on the  quotient group $\oo(n) / \oo(m)$ is the quotient topology from the group $\oo(n)/ \oo(k)$.
\item For any $n \le m \in \dz$, for any $a \in K^*$ the map of multiplication on $a$ is a homeomorphism
from $\oo(n)/\oo(m)$ to $a \oo(n) / a \oo(m)$.
\end{enumerate}
\end{defin}

\begin{nt}\em
Clearly, if $K$ is isomorphic to one of the fields from~\eqref{examples}  (with the fixed isomorphisms), then $K$ is endowed with  a partial topology.
Besides, if $K$ appears from an algebraic or arithmetis surface, then  the construction of $K$ uses composition of projective and inductive limits,
see, e.g.,~\cite{Osip1}.
Therefore, in this case  $K$ possesses the natural topology after these limits, and $K$ is endowed with  a partial topology given by  the quotient and induced topology on $\oo(n)/\oo(m)$.
Moreover, if $K$ is a two-dimensional local field with  residue field $\bar{K}$ which is isomorphic to the field $\df_q((u))$, then on $K$ there is a unique natural topology, see~\cite[\S~7.1-7.2]{K}.\footnote{We note that in the $n$-dimensional case, it was proved in~\cite[Prop.~2.1.21, Prop.~3.3.6]{Y} that the natural topology on the field isomorphic to $k((t_1))  \ldots ((t_n))$, where $\mathop{{\rm char}}  k  >0 $ and $k$ is perfect, is unique and this topology  coincides
with a topology which naturally appears when the field comes from an $n$-dimensional algebraic variety.} And  in this case $K$ is also endowed with  a partial topology given by the quotient and induced topology on $\oo(n)/\oo(m)$.
\end{nt}

For any $\alpha_1 \le  \alpha_2 \in \breve{\Gamma}_K$ there are $n \le m  \in \dz$ such that
$$
\oo(n)  \supset \oo'(\alpha_1)  \supset \oo'(\alpha_2) \supset \oo(m ) \, \mbox{.}
$$
Therefore the group $\oo'(\alpha_1)  / \oo'(\alpha_2)$
is locally compact with the induced and quotient topology from $\oo(n)/\oo(m)$ (this topology does not depend on the choice of $n,m$).
Thus, for any $\alpha_1 \le  \alpha_2 \le  \alpha_3 \in \breve{\Gamma}_K$ we have a short exact sequence of locally compact Abelian groups
\begin{equation} \label{short-sequence}
0 \lrto  \oo'(\alpha_2)  / \oo'(\alpha_3)  \lrto \oo'(\alpha_1)  / \oo'(\alpha_3)  \lrto \oo'(\alpha_1)  / \oo'(\alpha_2)  \lrto 0  \, \mbox{,}
\end{equation}
where the subgroup $\oo'(\alpha_2)  / \oo'(\alpha_3)$ is closed in $\oo'(\alpha_1)  / \oo'(\alpha_3)$ and has the induced topology, and the group
$\oo'(\alpha_1)  / \oo'(\alpha_2)$ has the quotient topology.

\subsection{Spaces $\D_{\alpha}(K)$ and $\D'_{\alpha}(K)$} \label{Spaces}
Let $K$ be a two-dimensional local field with a partial topology.

We recall constructions of function and distribution spaces $\D_{\alpha}(K)$ and $\D'_{\alpha}(K)$, where $\alpha \in \breve{\Gamma}_K$,
from~\cite[\S~5]{OsipPar1} and~\cite[\S~7]{OsipPar1'}.  But we will use another notation, see more explanation in Remark~\ref{notation} below.

For any Abelian totally  disconnected locally compact group $H$ let $\D(H)$ be the $\dc$-vector space of $\dc$-valued locally constant functions with compact support on $H$.
Let
\begin{equation}  \label{dual_1}
\D'(H) = \Hom\nolimits_{\dc} (\D(H), \dc)
\end{equation}
be the space of distributions on $H$.

For any $\alpha_1, \alpha_2 \in \breve{\Gamma}_K$ we define a  one-dimensional $\dc$-vector space $\mu(\alpha_1, \alpha_2)$:
\begin{eqnarray*}
\mu(\alpha_1, \alpha_2) = \mu(\oo'(\alpha_2)/ \oo'(\alpha_1)) \, \quad \mbox{if} \quad \alpha_1 \ge \alpha_2 \, \mbox{,} \\
\mu(\alpha_1, \alpha_2) = \mu(\alpha_2, \alpha_1)^*
=  \Hom\nolimits_{\dc} ( \mu(\alpha_2, \alpha_1)  , \dc)
\, \quad \mbox{if} \quad \alpha_1 \le \alpha_2 \, \mbox{,}
\end{eqnarray*}
where $\mu(\oo'(\alpha_2)/ \oo'(\alpha_1))$ is the space of all invariant $\dc$-valued measures (including the zero) on the locally compact Abelian group $\oo'(\alpha_2)/ \oo'(\alpha_1)$.
{ (We note also that the space $\mu(\oo'(\alpha_2)/ \oo'(\alpha_1))$ is canonically embedded into the space
$\D'(\oo'(\alpha_2)/ \oo'(\alpha_1))$ via  integration of a function against a measure.)}

From short exact sequence~\eqref{short-sequence} we have a canonical isomorphism for any \linebreak ${\alpha_1 \le  \alpha_2 \le  \alpha_3 \in \breve{\Gamma}_K}$:
\begin{equation}  \label{measures}
\mu(\alpha_1, \alpha_2) \otimes_{\dc} \mu(\alpha_2, \alpha_3)  \lrto \mu(\alpha_1, \alpha_3)  \, \mbox{.}
\end{equation}

We fix $\alpha \in \breve{\Gamma}_K$.

For any $\alpha_1 \le  \alpha_2 \le  \alpha_3 \in \breve{\Gamma}_K$ from short exact sequence~\eqref{short-sequence} we have a canonical map called direct image, which is given by integration { (along fibres)} of a function against a measure:
$$
\D(\oo'(\alpha_1) / \oo'(\alpha_3)) \otimes_{\dc} \mu(\alpha_3, \alpha_2) \lrto \D(\oo'(\alpha_1) / \oo'(\alpha_2)) \, \mbox{.}
$$
Hence and from~\eqref{measures} we have a map:
$$
\D(\oo'(\alpha_1) / \oo'(\alpha_3)) \otimes_{\dc} \mu(\alpha_3, \alpha) \lrto \D(\oo'(\alpha_1) / \oo'(\alpha_2)) \otimes_{\dc} \mu(\alpha_2, \alpha) \, \mbox{.}
$$
{ Besides, by restriction of functions we have a canonical map called inverse image:
$$
\D(\oo'(\alpha_1) / \oo'(\alpha_3))  \lrto \D(\oo'(\alpha_2) / \oo'(\alpha_3))  \, \mbox{.}
$$
 This gives a map:
 $$
\D(\oo'(\alpha_1) / \oo'(\alpha_3)) \otimes_{\dc} \mu(\alpha_3, \alpha) \lrto \D(\oo'(\alpha_2) / \oo'(\alpha_3)) \otimes_{\dc} \mu(\alpha_3, \alpha) \, \mbox{.}
$$
}

On the dual side, we have analogous maps with the spaces $\D'(\cdot)$. This leads to the following well-defined
function and distribution spaces
 after passing to projective or inductive limits:
\begin{equation} \label{D}
\D_{\alpha}(K)= \mathop{\lim_{\longleftarrow}}_{\alpha_1 \in \breve{\Gamma}_K } \, \mathop{\lim_{\longleftarrow}}_{\alpha_1 \le \alpha_2 \in \breve{\Gamma}_K} \D(\oo'(\alpha_1) / \oo'(\alpha_2)) \otimes_{\dc} \mu(\alpha_2, \alpha)  \, \mbox{,}
\end{equation}
\begin{equation}  \label{D'}
\D'_{\alpha}(K)= \mathop{\lim_{\lrto}}_{\alpha_1 \in \breve{\Gamma}_K } \, \mathop{\lim_{\lrto}}_{\alpha_1 \le \alpha_2 \in \breve{\Gamma}_K} \D'(\oo'(\alpha_1) / \oo'(\alpha_2)) \otimes_{\dc} \mu(\alpha, \alpha_2)  \, \mbox{.}
\end{equation}

Obviously, for any $\alpha, \beta \in \breve{\Gamma}_K$ there are canonical isomorphisms
\begin{equation}  \label{change}
\D_{\alpha}(K)  \otimes_{\dc} \mu(\alpha, \beta) \lrto \D_{\beta}(K)  \, \mbox{,} \qquad
\D'_{\alpha}(K)  \otimes_{\dc} \mu( \beta, \alpha) \lrto \D'_{\beta}(K)  \, \mbox{.}
\end{equation}
and a non-degenerate pairing:
\begin{equation}  \label{pairing-D}
\D_{\alpha}(K)  \times \D'_{\alpha}(K)  \lrto \dc \, \mbox{.}
\end{equation}

Similarly, for any $\gamma \in \breve{\Gamma}_K$ and $\alpha \ge \gamma$ we have function and distribution spaces on the space~$\oo'(\gamma)$:
$$
\D_{\alpha}( \oo'(\gamma)  )=  \mathop{\lim_{\longleftarrow}}_{\alpha_2 \ge \gamma  \in \breve{\Gamma}_K} \D(\oo'(\gamma) / \oo'(\alpha_2)) \otimes_{\dc} \mu(\alpha_2, \alpha)  \, \mbox{,}
$$
\begin{equation}  \label{d_al}
\D'_{\alpha}(\oo'(\gamma))= \mathop{\lim_{\lrto}}_{\alpha_2  \ge \gamma  \in \breve{\Gamma}_K} \D'(\oo'(\gamma) / \oo'(\alpha_2)) \otimes_{\dc} \mu(\alpha, \alpha_2)  \, \mbox{.}
\end{equation}

\subsection{Analogs of characteristic functions (or distributions) on $K$}  \label{charcteristic}
We keep all assumptions from section~\ref{Spaces}. We fix $\alpha \in  \breve{\Gamma}_K$.

For any $\gamma \in \breve{\Gamma}_K$,  $b \in \mu(\alpha, \gamma)$ we construct an element $\delta_{\gamma, b} \in  \D'_{\alpha}(K)  $, which can be seen for $\oo'(\gamma)$ as an analog of the characteristic function for a fractional ideal in one-dimensional case.

There are canonical maps (inverse and direct images):
$$
\Lambda^* \; : \; \D_{\alpha} (K) \otimes_{\dc}    \mu (\alpha, \max(\alpha, \gamma))  \lrto \D_{\max(\alpha, \gamma)} (\oo'(\gamma))  \, \mbox{,}
$$
$$
\Lambda_* \; : \;  \D'_{\max(\alpha, \gamma)} (\oo'(\gamma))    \otimes_{\dc}   \mu (\alpha, \max(\alpha, \gamma))  \lrto   \D'_{\alpha}(K)  \, \mbox{,}
$$
where these maps are defined either as the map from  the projective  limit to the projective limit with respect to  the subset of indices or as the map from the inductive  limit  with respect to the subset of indices to the inductive limit with respect to the full set of indices together with tensor product on  the one-dimensional $\dc$-vector space  $\mu (\alpha, \max(\alpha, \gamma))$ as in formulas~\eqref{change}.
Besides, the maps $\Lambda^*$ and $\Lambda_*$ are conjugate with respect to the pairing~\eqref{pairing-D} and its analog on $\oo'(\gamma)$.
(More on direct and inverse images $\Lambda_*$ and $\Lambda^*$ it  is written in~\cite[\S~5.6-5.7]{OsipPar1}.)

According to~\cite[\S~5.9.1]{OsipPar1},  for any element $c \in \mu(\max(\alpha, \gamma), \gamma)$
there is an element ${{\bf 1}_c \in  \D'_{\max(\alpha, \gamma)} (\oo'(\gamma)) }$ which is canonically constructed in the following way. We consider formula~\eqref{d_al} with the change of $\alpha$ to $\max(\alpha, \gamma)$ and put in this formula $\alpha_2 = \gamma$. Then we insert
$$
1 \in
\D'(\oo'(\gamma)/ \oo'(\gamma)) = \dc \, \qquad \mbox{and}  \qquad
 c \in \mu(\max(\alpha, \gamma), \gamma)
$$
in formula~\eqref{d_al}:
$$
{{\bf 1}_c = 1 \otimes c \in   \D'(\oo'(\gamma)/ \oo'(\gamma))  \otimes_{\dc}  \mu(\max(\alpha, \gamma), \gamma)    \subset \D'_{\max(\alpha, \gamma)} (\oo'(\gamma)) }  \, \mbox{.}
$$

Now for $d \in \mu (\alpha, \max(\alpha, \gamma))$ and $b = d \otimes c \in \mu(\alpha, \gamma)$ we define an analog of the characteristic function
for $\oo'(\gamma)$:
$$
\delta_{\gamma, b}  = \Lambda_*({\bf 1}_{c} \otimes d)  \in \D'_{\alpha}(K) \, \mbox{.}
$$

More direct definition of the element $\delta_{\gamma, b}$ can be given as following. We consider $\alpha_1 = \alpha_2 = \gamma$
in formula~\eqref{D'}.
Then we put $1 \in \D'(\oo'(\gamma)/ \oo'(\gamma)) = \dc$ and  $b \in \mu(\alpha, \gamma)$ in this formula. We have
$$
\delta_{\gamma, b} = 1 \otimes b  \in \D'(\oo'(\gamma)/ \oo'(\gamma)) \otimes_{\dc}  \mu(\alpha, \gamma)  \subset \D'_{\alpha}(K)  \, \mbox{.}
$$

\begin{nt} \em
The construction of the element $\delta_{\gamma, b}$ can be generalized to the construction of an element
$\delta_{\gamma,\lambda, f,b }$,  where $\gamma$ and $b$ are as above, i.e.  $\gamma \in \breve{\Gamma}_K$,  $b \in \mu(\alpha, \gamma)$, and $\lambda \ge \gamma \in \breve{\Gamma}_K$, and  $f $ is any locally constant function on the locally compact Abelian group $\oo'(\gamma)/ \oo'(\lambda)$.

Indeed, we fix any non-zero element
$$h \in \mu (\lambda, \gamma)  \subset \D'(\oo'(\gamma)/ \oo'(\lambda))   \, \mbox{.}$$
Then $f_h \in \D'(\oo'(\gamma)/ \oo'(\lambda))$ is given as the linear functional
$$g \longmapsto h(fg) \,  \mbox{,} \quad g \in \D(\oo'(\gamma)/ \oo'(\lambda))  \, \mbox{.}$$
 Besides, $h^{-1} \in \mu (\gamma, \lambda)$  and $h \otimes h^{-1} =1 \in \dc$.
 Now we define
$$
\delta_{\gamma,\lambda, f,b} = f_h \otimes (b \otimes h^{-1}) \in \D'(\oo'(\gamma)/ \oo'(\lambda)) \otimes_{\dc}  \mu(\alpha, \lambda) \subset \D'_{\alpha}(K)  \, \mbox{,}
$$
and the element $\delta_{\gamma,\lambda, f,b}$ does not depend on the choice of $h$.
Then we have $$
\delta_{\gamma,b} = \delta_{\gamma,\lambda, f,b}
$$
for any $\lambda \ge \gamma$ and where $f =1$ is the constant function.
\end{nt}

\bigskip

Now for any subgroup $E \subset K$ such that
\begin{itemize}
\item for any $n \in \dz$  the subgroup $$(E \cap \oo(n))/ (E \cap \oo(n+1))   \subset \oo(n)/\oo(n+1)$$
is open compact,
\item  and also  $$E = \mathop{\lim\limits_{\lrto}}\limits_{n} \mathop{\lim\limits_{\longleftarrow}}\limits_{m \ge n } ( E \cap \oo(n))/( E \cap \oo(m))$$
\end{itemize}
we define a characteristic function $\delta_E  \in \D_{\alpha}(K)$.

For any $\alpha_2 \ge \alpha_1 \in  \breve{\Gamma}_K$
we define an element
$$c_{\alpha_2, \alpha_1} \in \mu(\alpha_2, \alpha_1) = \mu(\oo'(\alpha_1) / \oo'(\alpha_2))$$
 by the following rule: the measure given by $c$
of the open compact group
$${(\oo'(\alpha_1) \cap E) /(\oo'(\alpha_2) \cap E) }$$ is equal to $1$.
For any $\alpha_2 \le \alpha_1 \in  \breve{\Gamma}_K$ we define $c_{\alpha_2, \alpha_1} = c_{\alpha_1, \alpha_2}^{-1} \in \mu(\alpha_2, \alpha_1)$.

For any $\alpha_2 \ge \alpha_1 \in  \breve{\Gamma}_K$ let ${\bf 1}_{\alpha_1, \alpha_2}$ be a function from $\D(\oo'(\alpha_1)/ \oo'(\alpha_2))$
which is equal to $1$ on the subgroup ${(\oo'(\alpha_1) \cap E) /(\oo'(\alpha_2) \cap E) }$ and is equal to $0$ on the complement to this subgroup in the group $\oo'(\alpha_1)/ \oo'(\alpha_2)$.

Now to define $\delta_E$ we put in formula~\eqref{D} elements
$$
{\bf 1}_{\alpha_1, \alpha_2} \otimes c_{\alpha_2, \alpha}  \, \mbox{.}
$$
for any $\alpha_1 \ge \alpha_2$. This is well-defined.

We note that the element $\delta_E$ can be also obtained as the direct image of some element, see more in~\cite[\S~5.9.2]{OsipPar1}.

\subsection{Invariants and coinvariants: definitions} \label{inv-sec}
We keep all assumptions from sections~\ref{Spaces} and~\ref{charcteristic}.

We note that the spaces $\D_{\alpha}(K)$ and $\D'_{\alpha}(K)$ can be rewritten as:
\begin{equation}  \label{D-new}
\D_{\alpha}(K)= \mathop{\lim_{\longleftarrow}}_{n \in \dz } \, \mathop{\lim_{\longleftarrow}}_{m \ge n } \D(\oo(n) / \oo(m)) \otimes_{\dc} \mu(m, \alpha) \, \mbox{,}
\end{equation}
\begin{equation} \label{D'-new}
\D'_{\alpha}(K)= \mathop{\lim_{\lrto}}_{n \in \dz} \, \mathop{\lim_{\lrto}}_{m \ge n} \D'(\oo(n) / \oo(m)) \otimes_{\dc} \mu(\alpha, m) \, \mbox{,}
\end{equation}
where an integer $m \in \dz$ is also considered as an element  $(m , - \infty) \in \breve{\Gamma}_K$.

The group $\oo_K^*$ acts on $K$ by multiplication. Hence it acts on the spaces $\D(\oo(n) / \oo(m))$, $\D'(\oo(n) / \oo(m))$ and $\mu(n,m)$
in the standard way. Therefore after passing to limits we obtain at once an action of $\oo_K^*$ on the $\dc$-vector spaces $\D_{\alpha}(K)$
and $\D'_{\alpha}(K)$ for $\alpha \in \dz$. (Moreover, using isomorphisms~\eqref{change}, it is easy to construct the action of  $\oo_K^*$  on these spaces for any $\alpha \in  \breve{\Gamma}_K$.)

Besides, we obtain that the group $(\oo'_K)^*$ acts on the $\dc$-vector spaces $\D_{\alpha}(K)$
and $\D'_{\alpha}(K)$ for any $\alpha \in  \breve{\Gamma}_K$.

We note that $(\oo_K')^*$ acts trivially on the one-dimensional $\dc$-vector space $\mu(\alpha_1, \alpha_2)$ for any $\alpha_1, \alpha_2 \in
\breve{\Gamma}_K$. Indeed, from formula~\eqref{measures}
it is enough to see it only in the cases:  $\mu(n, n+1)$, where $n \in \dz$, and $\mu(\alpha_1, \alpha_2)$, where
$\breve{\pi}(\alpha_1) = \breve{\pi}(\alpha_2)$.

We have that the transition maps in limits in formulas~\eqref{D-new}-\eqref{D'-new} are $(\oo_K')^*$-equivariant. Besides, these maps are surjective
in formula~\eqref{D-new} and injective in formula~\eqref{D'-new}. This gives us the following proposition on the spaces of invariants and coinvariants.

\begin{prop}  \label{prop1}
For any $\alpha \in \breve{\Gamma}_K$ we have
$$
\D'_{\alpha}(K)^{(\oo_K')^*} = \mathop{\lim_{\lrto}}_{n \in \dz} \, \mathop{\lim_{\lrto}}_{m \ge n} \D'(\oo(n) / \oo(m))^{(\oo_K')^*} \otimes_{\dc} \mu(\alpha, m)  \, \mbox{,}
$$
\begin{equation}  \label{inv-form}
\D_{\alpha}(K)_{(\oo_K')^*}= \mathop{\lim_{\longleftarrow}}_{n \in \dz } \, \mathop{\lim_{\longleftarrow}}_{m \ge n } \D(\oo(n) / \oo(m))_{(\oo_K')^*} \otimes_{\dc} \mu(m, \alpha)  \, \mbox{.}
\end{equation}
\end{prop}
{ \begin{proof}
We have to use the facts which are given before this proposition. We add only that for the proof of formula~\eqref{inv-form} we consider an exact sequence for any integers $m \ge n$:
\begin{equation} \label{pr}
0 \lrto V_{n,m}  \lrto  \D(\oo(n) / \oo(m))  \otimes_{\dc} \mu(m, \alpha) \lrto  \D(\oo(n) / \oo(m))_{(\oo_K')^*}  \otimes_{\dc} \mu(m, \alpha) \lrto 0 \, \mbox{.}
\end{equation}
Now the double projective limits from the right hand sides of formulas~\eqref{D-new}   and~\eqref{inv-form} can be rewritten as ordinary projective limits over the set of indices  $m >0$, and where $n =-m$.
We note that the natural maps  $V_{-k,k} \lrto V_{-m,m}$ are surjective when $k >m$. Therefore the Mittag-Leffler condition for the projective system $V_{-m,m}$ is satisfied. Hence, after taking the projective limit of exact triples~\eqref{pr} over $m >0$ and where $n=-m$,  we obtain formula~\eqref{inv-form}.
\end{proof}}

\begin{nt} \em
Clearly, elements $\delta_{\gamma, b}$ and, more generally, $\delta_{\gamma,\lambda, f,b }$ when $f$ is an $(\oo_K')^*$-invariant function (see section~\ref{charcteristic})
belong to the space $\D_{\alpha}'(K)^{(\oo_K')^*}$.
\end{nt}

\begin{nt}  \label{notation} \em
We presented in section~\ref{Funct_distr} some facts from~\cite{OsipPar1} and~\cite{OsipPar1'} with some simplifications in application to a two-dimensional local field $K$.
We note that in this papers elements of harmonic analysis were constructed in much more generality: on objects of the category $C_2$
or $C_2^{\rm ar}$. (The first field from~\eqref{examples} is an object of $C_2$ and of $C_2^{\rm ar}$ and the second and third fields from~\eqref{examples} are objects of category~$C_2^{\rm ar}$.) Objects of these categories are filtered groups with additional conditions, see~\cite[\S~5]{OsipPar1} and~\cite[\S~5]{OsipPar1'}.
In the case of field $K$ this filtration can be given, for example,  by $\oo'_K$-modules $\oo'(\alpha)$, where $\alpha \in \breve{\Gamma}_K$,
or by $\oo_K$-modules $\oo(n)$, where $n \in \dz$.  The constructions from~\cite{OsipPar1} and~\cite{OsipPar1'} used inductive  and projective limits of some spaces depending on these filtrations. Besides, we used in these papers the spaces of virtual measures which depend on elements of the filtration. In our case, the spaces of virtual measures   are the spaces $\mu(\alpha_1, \alpha_2)$ from section~\ref{Spaces}, and they are equal either to the space $\mu(\oo'(\alpha_1)/ \oo'(\alpha_2))$ or to the space $\mu(\oo'(\alpha_2)/ \oo'(\alpha_1))^*$
from section~\ref{Spaces}.
\end{nt}

\subsection{Invariants and coinvariants: duality}  \label{inv-dual}
We can say more on spaces of invariants and coinvariants.

Let $L$ be a one-dimensional local field with finite  residue field and $K$ be a two-dimensional local field with a partial topology. Let $\alpha \in \breve{\Gamma}_K$.

From~\eqref{dual_1} and~\eqref{pairing-D} we have  canonical pairings:
$$
\D(L)_{\oo_L^*}  \times \D'(L)^{\oo_L^*}  \lrto \dc  \, \mbox{,}
$$
\begin{equation}  \label{dual_inv_two}
\D_{\alpha}(K)_{(\oo_K')^*}  \times \D'_{\alpha}(K)^{(\oo_K')^*} \lrto \dc  \, \mbox{.}
\end{equation}

We recall the following easy lemma.

\begin{lemma}  \label{lem1}
Let $G$ ba a group, which acts on a vector space $V$ over a field $E$. Then there is a canonical isomorphism:
$$
 \Hom\nolimits_E(V_G, E)  \simeq \Hom\nolimits_E (V, E)^G\, \mbox{.}
$$
\end{lemma}

From this lemma we obtain at once that

\begin{equation}  \label{dual_inv_one}
\Hom\nolimits_{\dc} ( \D(L)_{\oo_L^*}, \dc) \simeq  \D'(L)^{\oo_{L}^*}   \, \mbox{.}
\end{equation}

\begin{prop}
The pairing~\eqref{dual_inv_two}  is non-degenerate.
\end{prop}
\begin{proof}
From Lemma~\ref{lem1} we have that for any integer $m \ge n$ the following pairing is non-degenerate:
$$
\D(\oo(n)/\oo(m))_{(\oo'_K)^*}  \times    \D'(\oo(n)/\oo(m))^{(\oo'_K)^*}  \lrto 0  \, \mbox{.}
$$
Besides, the transition maps in formulas for $\D'_{\alpha}(K)^{(\oo_K')^*}$ and $\D_{\alpha}(K)_{(\oo_K')^*}$ in Proposition~\ref{prop1} are injective and surjective correspondingly, since the functor of taking invariants is left exact and the functor of taking of coinvariants is right exact. Hence we obtain that
the pairing~\eqref{dual_inv_two} is non-degenerate.
\end{proof}

\begin{nt}{\em
We note that $\D_{\alpha}(K)^{(\oo_K')^*}=0 $, see~\cite[Prop.~4]{OsipPar2}.
}
\end{nt}

\section{Function and distribution spaces on free Abelian groups of ranks $1$ and $2$}  \label{dis}

\subsection{Rank~$1$}
In this section we describe function and distribution spaces on a
$\dz$-torsor $T$
and connect these spaces with invariant and coinvariant spaces of function and distribution spaces on a one-dimensional  vector space  $N$
over a one-dimensional local field with a finite residue field when the $\dz$-torsor $T$ naturally appears from the space  $N$.

\subsubsection{The space $\D_+(T)$}  \label{prop-int}

Let $T$ be a $\dz$-torsor. Then $T$ is naturally linearly ordered. We define a $\dc$-vector space $\D_+(T)$ as the space of functions from $T$ to $\dc$ with  conditions:
\begin{multline*}
\D_+(T) = \{c \, :  \,   x \longmapsto c_x \, \mbox{,} \quad x \in T \, \mbox{,} \quad c_x \in \dc   \quad \mbox{such that} \quad
c_y =0 \quad  \mbox{when} \quad   y \le x_0 \quad  \mbox{and} \\ \quad  c_x = c_{x+1} \quad  \mbox{when} \quad  x \ge x_1 \quad  \mbox{(the elements $x_0$ and $x_1$ depend on $c$)} \}  \, \mbox{.}
\end{multline*}

For any $x \in T$ we consider $\delta_{\ge x} \in \D_+(T)$:
$$
\delta_{\ge x} (y) = 0 \quad \mbox{if} \quad y < x \, \mbox{, and} \quad \delta_{\ge x} (y) =1 \quad \mbox{if} \quad y \ge x \, \mbox{.}
$$
The set of elements $\delta_{\ge x}$ (where $x \in T$) is a basis in the $\dc$-vector space  $\D_+(T)$.

\begin{nt} \em
For $T = \dz$  the space $\D_+(\dz)$ was used in~\cite{Par3}.
\end{nt}

 \medskip

For the applications it is important the following proposition.
\begin{prop} \label{prop_dual}
Let $L$ be a one-dimensional local field with  finite  residue field $\df_q$. Let $\mu $ be a (non-zero) Haar measure on $L$ and $u$ be a local parameter in $L$. Then the following diagram is commutative:
$$
\xymatrix{
\D(L)^{\oo_L^*} \,  \ar@{^{(}->}[r]
\ar[d]^(.45){i} & \,
\D(L)
\ar@{->>}[r]^(.45)v &
\D(L)_{\oo_L^*}
\ar[d]_(.45){j} \\
\D_+(\dz)
\ar@{=}[rr]  &&
\, \D_+(\dz)  \; \mbox{,}
}
$$
where the maps in upper row are natural embedding and surjection correspondingly,
the map $i$ is defined as
$$
\D(L)^{\oo_L^*}  \ni f  \longmapsto i(f) \in \D_+(\dz) \, \mbox{,} \quad  i(f) \, : \,   n \longmapsto f(u^n)   \, \mbox{,}
$$
and the map $j$ is defined as
$$
\D(L)_{\oo_L^*}  \ni g  \longmapsto j(g) \in \D_+(\dz)
\, \mbox{,} \quad  j(g) \, : \,   n \longmapsto \frac{1}{\mu(O_L^*(n))} \, \int_{\oo^*_L(n)} g(x) \, d\mu(x)    \, \mbox{,}
$$
where $\oo^*_L(n) = m_L^n \setminus m_L^{n-1}$, $m_L$ is the maximal ideal in the discrete valuation ring $\oo_L$ (cf.  section~\ref{can-dec}).

Besides, we have that the maps $i$ and $j$ do not depend on the choice of $u$ and $\mu$ correspondingly,
and  the maps $i$ and $j$ are isomorphisms.
\end{prop}
\begin{nt} \em
We note that
$\mu(\oo^*_L(n)) = \mu(\oo_L)(q^{-n} - q^{-n-1})$
for any $n \in \dz$.
\end{nt}
{ \begin{proof}
For any $n \in \dz$ let $\delta_{m_L^n} \in \D(L)^{\oo_L^*} $ be the characteristic function of the fractional ideal $m_L^n$ of $L$.
We have $i(\delta_{m_L^n})= \delta_{\ge n}$. Therefore the map $i$ is surjective. Clearly, this map is also injective.

We note that the map $j$ is well-defined, since for any $r \in \oo_L^*$ we have
$$
\int_{\oo^*_L(n)} g(x) \, d\mu(x) = \int_{\oo^*_L(n)} r(g)(x) \, d r(\mu) (x) = \int_{\oo^*_L(n)} r(g)(x) \, d\mu(x) \, \mbox{.}
$$
Besides, we have $j \circ v (\delta_{m_L^n})= \delta_{\ge n}$. Therefore the map $j$ is surjective.

Now we prove that the map $j$ is injective. Let $f \in \D(L)$ such that $j \circ v (f)=0$. We will prove that
$$f = \sum_{1 \le l \le s} r_l (f_l) - f_l   \, \mbox{,} \qquad \mbox{where} \quad  s \in \dn \ \mbox{,} \quad  \mbox{and} \quad
r_l \in \oo_L^*  \, \mbox{,} \quad  f_l \in \D(L) \quad   \mbox{for} \quad  1 \le l \le s \, \mbox{.} $$
 There are numbers
$k \le m \in \dz$ and the function $\tilde{f} : m_L^k / m_L^m  \lrto \dc$ such that
${f(x)= \tilde{f}(\bar{x})}$
for  $x \in m_L^k$, where $\bar{x}$ is the image of $x$ in $m_L^k/ m_L^m$,
 and $f(x)= 0$
for  $x \in L \setminus m_L^k$. It is easy to see that it is enough to prove that
$$
\tilde{f} = \sum_{1 \le l \le s} r_l (\tilde{f}_l) - \tilde{f}_l   \, \mbox{,} \; \: \mbox{where} \; \:  s \in \dn \ \mbox{,} \quad  \mbox{and} \quad
r_l \in \oo_L^*  \, \mbox{,} \quad  \tilde{f}_l : m_L^k/ m_L^m \lrto \dc \quad   \mbox{for} \quad  1 \le l \le s \, \mbox{.}
$$

The finite group $ m_L^k / m_L^m$ is the disjoint union of orbits
of the group $\oo_L^*$:
$$ m_L^k  / m_L^n = \coprod_{1 \le s \le m-k+1} O_s \, \mbox{.}
$$

For every $s$  from condition $j \circ v (f)=0$ we have
\begin{equation}  \label{sum}
\sum_{x \in O_s} \tilde{f}(x) =0  \, \mbox{.}
\end{equation}
For every $s$ we fix $y_s \in O_s$. From~\eqref{sum} we obtain
$$
\tilde{f} =  \sum_{1 \le s \le m-k+1}  \left( \sum_{x \in O_s} \tilde{f}(x) \left( \delta_{x} - r_{x,y_s}(\delta_{x})  \right) \right)  \, \mbox{,}
$$
where for any $x \in O_s$ an element $r_{x,y_s} \in \oo_L^*$ satisfies  $r_{x,y_s} x =  y_s$, and the function $\delta_{x} : m_L^k / m_L^m  \lrto \dc$ is equal to $1$ at $x$ and is equal to $0$ otherwise.

Therefore the map $j$ is injective. This all together gives the proof of the proposition.
\end{proof}}

\subsubsection{The space $\D_+'(T)$}

Let, as in section~\ref{prop-int}, $T$ be a $\dz$-torsor.

We define a $\dc$-vector space
$$
\D_+'(T)  = \Hom\nolimits_{\dc} (\D_+(T), \dc)  \, \mbox{.}
$$

\begin{nt} \label{isom} \em
From isomorphism~\eqref{dual_inv_one} and Proposition~\ref{prop_dual} we obtain a canonical isomorphism:
$$
\D'(L)^{\oo_L^*} \simeq \D_+'(\dz)  \, \mbox{.}
$$
\end{nt}

\medskip

We describe now the space $\D_+'(T)$  more explicitly.

We note that
\begin{equation}  \label{ind_limit}
\D_+(T) = \mathop{\lim_{\lrto}}_{x  \in T} \D(T_x)  \, \mbox{,}
\end{equation}
where $T_x = \{ y \in T : y \le x \}$, and $\D(T_x)$
 is the space of $\dc$-valued functions on $T$ with the finite support. And if $x_1 \le x_2$, then the transition map $T_{x_1}  \hookrightarrow T_{x_2}$
 in the inductive limit~\eqref{ind_limit}
is given as
$$
\D(T_{x_1}) \ni f  \longmapsto  g \in  \D(T_{x_2})  \; : \;
\left \{
\begin{array}{l}
g(y) = f(y) \quad \mbox{for} \quad y \le x_1 \, \mbox{,}\\
g(y)= f(x_1) \quad \mbox{for} \quad x_2 \ge y \ge x_1
\, \mbox{.}
\end{array}
\right.
$$

Therefore we have
\begin{equation}  \label{pro}
\D'_+(T) = \mathop{\lim_{\longleftarrow}}_{x  \in T} \Hom\nolimits_{\dc} (\D(T_x), \dc)  \, \mbox{.}
\end{equation}
The space $\Hom_{\dc} (\D(T_x) , \dc)$ is naturally identified with the space of all $\dc$-valued functions on $T_x$. In other words,
any element from  $\Hom\nolimits_{\dc} (\D(T_x), \dc)$ is a collection of elements
$$
a_{y,x} \in \dc \, \mbox{,} \quad \mbox{where} \quad T_x \ni y \le x  \, \mbox{,}
$$
and the pairing with elements from $\D(T_x)$ is the sum over $y \le x \in T$ of pairwise products of elements corresponding to $y$. Then
the transition maps in projective limit~\eqref{pro} are the following:
\begin{equation}  \label{form_pro}
a_{y,x} = a_{y, x+1}  \quad \mbox{if} \quad y < x \, \mbox{,}
\quad \mbox{and} \quad
a_{x,x}= a_{x,x+1}  + a_{x+1, x+1} \, \mbox{.}
\end{equation}

We note that elements $a_{x,x}$ (where $x \in T$) uniquely define an element from $\D'_+(T)$. Indeed, for fixed $x \in T$ we can obtain from~\eqref{form_pro}:
$$
a_{x,x+1} = a_{x,x} - a_{x+1, x+1} \, \mbox{,} \quad
a_{x,x+2} = a_{x,x+1} \, \mbox{,}
\quad
a_{x, x+3}= a_{x,x+2} \ldots
$$
Thus, we can identify
$$
\D'_+(T) \simeq \left\{ \{a_{x,x} \}_{x \in T} \right\}  \, \mbox{.}
$$

{ We recall that the set of elements $\delta_{\ge x}$ (where $x \in T$) is a basis in the $\dc$-vector space $\D_+(T)$, see section~\ref{prop-int}.} A number $a_{x,x} \in \dc$ is the value of the pairing of the  corresponding element from $\D_+'(T)$ with the element $\delta_{\ge x}$.

{ We note that for a torsor $J$ over  an Abelian locally compact group $H$ the $\dc$-vector spaces $\D(J)$, $\D'(J)$ and $\mu(J)$ are well-defined. }

\begin{prop} \label{delta}
Let $L$ be a one-dimensional local field with  finite residue field~$\df_q$. Let $N$ be a one-dimensional vector space over $L$.
Let $$T = (N \setminus 0) \times^{L^*} (L^*/ \oo_L^*)   $$  be a $\dz$-torsor. Let
${\delta_{(0)} \in \D'(N)}$ be defined as $\delta_{(0)}(f)= f(0)$, where $f \in \D(N)$, and let
 $\mu \in \D'(N)$ be an invariant $\dc$-valued
measure on $N$.
Then under an isomorphism (which is given as  in Remark~\ref{isom})
\begin{equation}  \label{iso-T}
\Lambda \; : \;\D'(N)^{\oo_L^*} \simeq \D_+'(T)
\end{equation}
the element $\delta_{(0)}$ goes to the element
$$\delta_{+\infty} = \Lambda(\delta_{(0)}) \quad  : \quad a_{x,x}=1 \quad  \mbox{for any} \quad  x \in T \quad \mbox{and} \quad  a_{y,x}=0
 \quad
 \mbox{if}  \quad y < x \,\mbox{,} $$
the element $\mu$ goes to the element
$$
a_{x,x}= \mu(\oo_x) \quad  \mbox{for any} \quad  x \in T \quad  \mbox{and}  \quad a_{y,x} = (1 - q^{-1}) \mu(\oo_y) \, \mbox{,}
$$
where $\oo_x = b \, \oo_L $ with $\nu(b) =x$ for the canonical morphism $\nu: (N \setminus 0) \to T$.
\end{prop}
\begin{proof}
The isomorphism~\eqref{iso-T} is defined as in Remark~\ref{isom}: it is enough to fix an isomorphism $N \simeq L$.
To obtain the presentations for elements $\delta_{+\infty} = \Lambda( \delta_{(0)})$ and $\Lambda(\mu)$
it is enough to calculate their pairings with elements $\delta_{\ge x}$ (for $a_{x,x}$)
and with elements $\delta_{\ge y} - \delta_{\ge (y +1)}$ if $y < x$ (for~$a_{y,x}$).
\end{proof}

\begin{nt} \label{deform} \em
For any $r \in \dc^*$ we define an element $g_{r} \in  \D'_+(\dz) $ as
$$
a_{n,n} = r^{-n} \, \mbox{,} \quad  a_{m,n}= r^{-m}(1 -  r^{-1}) \quad \mbox{if} \quad m < n \, \mbox{.}
$$
Then under isomorphism $\D'_+(\dz) \simeq D'(L)^{\oo_L^*}$ we have
$g_{r}(\delta_{\oo_L})=1$  and  $u^*(g_{r})= r g_{r}$, where $\delta_{\oo_L} \in \D(L)$ is the characteristic function of $\oo_L$ and $u$ is a local parameter of $L$ which naturally acts on $\D'(L)$ (this action is induced by multiplication on the field $L$). The elements $\delta_{+ \infty} = \Lambda(\delta_{(0)})$ and $\Lambda(\mu)$ (with the property $\mu(\delta_{\oo_L})=1$) from proposition~\ref{delta}
are obtained when $r= 1$ and $r =q$ correspondingly.

From this point of view we obtain that $\delta_{(0)}$ is the deformation of $\mu$ through $r \in \dc^*$.
\end{nt}

\medskip

For any $\dz$-torsor $T$ let a $\dc$-vector space $\D(T)$ be the space of $\dc$-valued functions on $T$ with the finite support.
We define
$$
\D'(T) = \Hom\nolimits_{\dc}(\D(T), \dc)  \, \mbox{.}
$$

We have a natural  exact sequence:
$$
0 \lrto \D(T)  \lrto \D_+(T)  \stackrel{w}{\lrto} \dc \lrto 0 \, \mbox{,}
$$
where $w(f) = \delta_{+\infty}(f)$ for $f \in \D(T)$.

We consider a dual exact sequence:
$$
0 \lrto \dc \stackrel{w'}{\lrto} \D_+'(T)  \stackrel{s}{\lrto} \D'(T) \lrto 0  \, \mbox{,}
$$
where $w'(s)= s \cdot \delta_{+\infty}$
 for $s \in \dc$.
 The space $\D'(T)$ is naturally identified with  the space of all $\dc$-valued  functions on $T$.
 The map $s : \D_+'(T)  {\twoheadrightarrow} \D'(T) $ is described as:
 $$
 \{ a_{x,x}\}  \longmapsto \{ x \mapsto a_{x, x+1}  \} \, \mbox{.}
 $$
 In other words, the value of the function $s(\{a_{x,x} \})$ at $x \in T$ is equal to the number $a_{x,x+1} = a_{x,x} -a _{x+1, x+1}$, which is also equal to the value of the pairing of $\{a_{x,x} \}$ with the element $\delta_{\ge x} - \delta_{\ge(x+1)} $.

\subsection{Rank~$2$}
In this section we will  construct function and distribution spaces on a free Abelian group $\Gamma$ of rank $2$ such that if $\Gamma = \Gamma_{K}$, where  $K$ is a two-dimensional local field with a partial topology, then the corresponding function and distributions spaces are connected with  the coinvariant and invariant spaces: $\D_{\alpha}(K)_{(\oo'_K)^*}$ and $\D'_{\alpha}(K)^{(\oo'_K)^*}$ for $\alpha \in \breve{\Gamma}_K$.

\subsubsection{More on spaces $\D(\oo(n) /\oo(m))_{(\oo'_K)^*}$  and $\D'(\oo(n) /\oo(m))^{(\oo'_K)^*}$}  \label{reas}

First, we  give  more statements on the spaces
$$ \D(\oo(n) /\oo(m))_{(\oo'_K)^*} \qquad   \mbox{and} \qquad  \D'(\oo(n) /\oo(m))^{(\oo'_K)^*} \, \mbox{,}$$
where $n \le m \in \dz$.

Let $K$  be a  two-dimensional local field with a partial topology.

We consider  $n \le m \in \dz$. We define a map $\Phi$:
\begin{multline} \label{Phi}
\D(\oo(n) /\oo(m))_{(\oo'_K)^*} \otimes_{\dc} \mu(\oo(n)/ \oo(m))  \stackrel{\Phi}{ \lrto}  \\
\prod_{n \le k < m} \D(\oo(k)/ \oo(k+1))_{(\oo_K')^*}  \otimes_{\dc} \mu(\oo(n)/ \oo(k+1))
 \stackrel{\phi}{\simeq} \\
 \prod_{n \le k < m} \D_+(T(k))  \otimes_{\dc} \mu(\oo(n)/\oo(k+1))  \, \mbox{.}
\end{multline}
Here $T(k)$ is a $\dz$-torsor canonically constructed by $\oo(k)/ \oo(k+1)$  (see Proposition~\ref{delta}),  the isomorphism $\phi$
is induced by the product of canonical isomorphisms ${\D(\oo(k)/ \oo(k+1))_{(\oo'_K)^*} \simeq \D_+(T(k))}$ (see Proposition~\eqref{prop_dual}). The map $\Phi$ is
$$
\Phi = \prod_{n \le k < m } \tau_k  \, \mbox{,}
$$
where the map $\tau_k$ is the composition of the map
\begin{equation}  \label{map-1}
\D(\oo(n)/ \oo(m))_{(\oo'_K)^*}  \otimes_{\dc} \mu(\oo(n)/ \oo(m)) \lrto \D(\oo(k)/ \oo(m))_{(\oo'_K)^*}   \otimes_{\dc}  \mu(\oo(n)/ \oo(m))
\end{equation}
which is induced by the inverse image with the tensor product on the identity map  on  the space of measures, and the map
$$
 \D(\oo(k)/ \oo(m))_{(\oo'_K)^*}   \otimes_{\dc}  \mu(\oo(n)/ \oo(m))  \lrto \D(\oo(k)/ \oo(k+1))_{(\oo'_K)^*}   \otimes_{\dc}  \mu(\oo(n)/ \oo(k+1))
$$
which is induced by the direct image with the tensor product on the identity map  on  the space of measures.

\begin{nt} \em
We don't know the description of $\Ker \Phi$.
\end{nt}

Now we describe $\Image \Phi$. We define a {$\dc$-vector} subspace
$$
W \subset \prod_{n \le k < m} \D_+(T(k))  \otimes_{\dc} \mu(\oo(n)/\oo(k+1))
$$
in the following way. We consider elements
$$
f_k \otimes \mu_k  \subset \D_+(T(k))  \otimes_{\dc} \mu(\oo(n)/\oo(k+1)) \, \mbox{,} \quad  \mbox{where} \quad \mu_k \ne 0 \quad \mbox{and}
\quad
n \le k < m  \,\mbox{.}
$$
(The condition $\mu_k \ne 0$ is not a restriction, since $f \otimes 0 = 0 \otimes \mu$, where $\mu \ne 0$.)
Then an element
$$
\prod_{n \le k < m } f_k \otimes \mu_k
$$
belongs to $W$ if and only if
\begin{gather}
\delta_{+\infty} (f_k) = \eta_{k+1}(f_{k+1})  \, \mbox{,}   \label{cond}\\ \nonumber
\mbox{where} \quad \mu_{k+1} = \mu_k \otimes \eta_{k+1} \quad  \mbox{for any} \quad n \le k < m-1 \\  \nonumber
\mbox{and} \quad \eta_{k+1} \in \mu(\oo(k+1)/ \oo(k+2)) \subset \D'(\oo(k+1)/\oo(k+2))^{(\oo_{K}')^*} \simeq
\D_+'(T(k+1)) \, \mbox{.}
\end{gather}

\begin{prop}  \label{W}
We have
$$
\Image \left( \phi \circ \Phi \right) = W  \, \mbox{.}
$$
\end{prop}
{ \begin{proof}
To prove that $\Image \left( \phi \circ \Phi \right) \subset W$ it is enough to consider for any $k$, where ${n \le k \le m-2}$, the map
$$
\alpha_{k}  :
\D(\oo(n) /\oo(m))_{(\oo'_K)^*} \otimes_{\dc} \mu(\oo(n)/ \oo(m))  \lrto
 \D(\oo(k)/ \oo(k+2))_{(\oo_K')^*}  \otimes_{\dc} \mu(\oo(n)/ \oo(k+2))
$$
which is the composition of the map~\eqref{map-1}
 and the map
$$
 \D(\oo(k)/ \oo(m))_{(\oo'_K)^*}   \otimes_{\dc}  \mu(\oo(n)/ \oo(m))  \lrto \D(\oo(k)/ \oo(k+2))_{(\oo'_K)^*}   \otimes_{\dc}  \mu(\oo(n)/ \oo(k+2))
$$
 induced by the direct image with the tensor product on the identity map  on  the space of measures. Then the map $\tau_k \times \tau_{k+1}$
 factors over  the map $\alpha_{k}$. Hence we obtain conditions~\eqref{cond}.

 Now we prove that any element from $W$ is in the image of  $\phi \circ \Phi$.
 We prove it by induction on $m -n$. If $m-n =1$, then it is nothing to prove. We suppose that we have already proved it for
 $n_0=n$ and $m_0= m-1$ (or for $n_0 = n+1$ and $m_0 =m$). We note that the natural map
$$ \D(\oo(n) /\oo(m))_{(\oo'_K)^*} \otimes_{\dc} \mu(\oo(n_0)/ \oo(m))  \lrto
\D(\oo(n_0) /\oo(m_0))_{(\oo'_K)^*} \otimes_{\dc} \mu(\oo(n_0)/ \oo(m_0))
 $$
 is surjective.  Now we have to take the tensor product of this map with  the identity map on the  $1$-dimensional $\dc$-vector space $\mu(\oo(n) / \oo(n_0))$
 and add, if necessary,  a ``correction''  element constructed by means of the following lemma.

 \begin{lemma} \label{lem-int}
 We consider an exact sequence for any $v < s < t$:
 $$
 0 \lrto \oo(s)/ \oo(v)  \stackrel{\alpha}{\lrto} \oo(t)/ \oo(v) \stackrel{\beta}{\lrto} \oo(t)/ \oo(s)  \lrto 0 \, \mbox{.}
 $$
Let $\mu$ be a (non-zero) Haar measure on $\oo(s)/ \oo(v)$. Let $f_1 \in \D(\oo(s)/ \oo(v))$ such that $\int_{\oo(s)/ \oo(v)} f_1(x) d \mu(x) =0$,
and $f_2 \in \D(\oo(t)/ \oo(s))$ such that $f_2(0)=0$. Then there are $g_1$ and $g_2$ from $\D(\oo(t)/ \oo(v))$ such that
$\alpha^* (g_1) = f_1$, $\beta_*(g_1 \otimes \mu) =0$, and $\beta_*(g_2 \otimes \mu)= f_2 $, $\alpha^*(g_2)=0$. Here $\alpha^*$ is the inverse image and $\beta_*$ is the direct image.
\end{lemma}
 (It is  easy to prove this lemma if we consider locally constant functions with compact supports as functions which come from functions on finite Abelian groups as in the second part of the proof  of Proposition~\ref{prop_dual}, or as in~\cite[\S~4.2]{OsipPar1}.)
\end{proof}}

\medskip

We have also ``the dual picture'':
\begin{equation}  \label{Phi'}
  W^* = \Hom\nolimits_{\dc}(W,\dc)  \; \subset \;  \D'(\oo(n) /\oo(m))^{(\oo'_K)^*} \otimes_{\dc} \mu(\oo(n)/ \oo(m))^* \, \mbox{.}
\end{equation}
Besides, we have
$$
0 \lrto V  \lrto   \bigoplus_{n \le k < m} \D'_+(T(k))  \otimes_{\dc} \mu(\oo(n)/\oo(k+1))^*   \lrto  W^*  \lrto 0  \, \mbox{,}
$$
where $ \D'_+(T(k))  \simeq \D'(\oo(k)/ \oo(k+1))^{(\oo_K')^*} $ and the subspace $V$ is generated by elements:
$$
h_k \otimes \mu_k - z_{k+1} \otimes \lambda_{k+1}  \, \mbox{,}
\qquad
n \le k < m \, \mbox{,}  $$
where  $ h_l  \mbox{,} \; z_l \in \D'_+(T(l))$ and $ \mu_l \mbox{,} \; \lambda_l \in \mu(\oo(n)/\oo(l+1))^*
$,
with the following conditions
$$
h_k = \delta_{+ \infty}  \, \mbox{,} \qquad \mu_k = \eta_{k+1} \otimes \lambda_{k+1}  \, \mbox{,}
$$
$$
 \eta_{k+1} \in \mu(\oo(k+1)/ \oo(k+2))  \subset \D'(\oo(k+1)/ \oo(k+2))^{(\oo'_K)^*} \simeq \D'_+(T(k+1)) \, \mbox{,} \quad z_{k+1} = \eta_{k+1} \, \mbox{.}
$$

\subsubsection{Free Abelian groups of rank~$2$}  \label{discr_2}
Now we will use reasonings from section~\ref{reas} to define function and distribution spaces on free Abelian groups of rank~$2$.

Let $\Gamma$ be a group such that there is a central extension (which we fix):
\begin{equation} \label{centr}
0 \lrto \dz \lrto \Gamma \stackrel{\pi}{\lrto} \dz \lrto 0  \, \mbox{.}
\end{equation}
We note that there is always a section of the homomorphism $\pi$ and, therefore, there is a decomposition $\Gamma \simeq \dz \oplus \dz$,
but we don't fix this section. We fix only  $\pi$.

We consider a naturally linearly ordered set:
$$
\breve{\Gamma} = \Gamma \cup \pi(\Gamma) \, \mbox{,}
$$
where elements of $\breve{\Gamma}$ are pairs
$$
(n, p)  \in \Gamma \quad \mbox{or} \quad (n, -\infty) \, \mbox{,} \quad \mbox{where} \quad n \in \pi(\Gamma) \, \mbox{,} \quad p \in \pi^{-1}(n)  \, \mbox{.}
$$
We have a natural map $\breve{\pi} : \breve{\Gamma} \to \dz$ which extends the map $\pi$ (the map $\breve{\pi}$ restricted to  $\pi(\Gamma)$ is the identity map). For any elements $\alpha_1$ and $\alpha_2$ from $\breve{\Gamma}$
we have
$$
\alpha_1 < \alpha_2 \quad \mbox{if} \quad \breve{\pi}(\alpha_1) < \breve{\pi}(\alpha_2) \, \mbox{,} \quad \mbox{and} \quad
(n, -\infty ) < (n, p)  \, \mbox{,} \quad (n, p_1) < (n, p_2) \quad \mbox{if}  \quad p_1 < p_2 \, \mbox{.}
$$

\bigskip

For any two subsets $A$ and $B$ of a set $R$ we say that
$$
A \sim B   \qquad \mbox{iff} \qquad \mid A \setminus (A \cap B) \mid < \infty  \quad \mbox{and} \quad \mid B \setminus (A \cap B) \mid < \infty \, \mbox{.}
$$
These two conditions are equivalent to the condition
 $$
 \mid (A \cup B)  \setminus (A \cap B)  \mid < \infty  \, \mbox{.}
 $$

Let $(R, \tau)$ be a pair, where $R$ is any set and $\tau$ is a family of subsets in the set $R$ such that
\begin{itemize}
\item if $A_1$ and $A_2$ belongs to $\tau$, then $A_1 \sim A_2$,
\item and  if for $A \in \tau$ and for a subset $B$ of $R$ we have $A \sim B$, then $B$ belongs to $\tau$.
\end{itemize}

\begin{defin}
For $(R, \tau)$ as above,
a $\dz$-torsor $[R]_{\tau}$ consists of all maps $d  :  \tau \lrto \dz$ such that if ${A_1  \mbox{,} \, A_2 \in \tau}$ and ${A_1 \subset A_2 }$, then
$$
d(A_2)= d(A_1) + \mid A_2 \setminus A_1 \mid \, \mbox{.}
$$
For any $a \in \dz$ and $d \in [R]_{\tau}$ the action is:
$$
(a+d )(A)= d(A) +a  \, \mbox{,}
$$
where $A \in \tau$.
\end{defin}

We note that
if $R = \emptyset$, then $[R]= [R]_{\tau}= \dz$.

We have an evident property. We suppose that  $(R, \tau)$ is obtained from $(R_1, \tau_1)$ and $(R_2,\tau_2)$ in the following way: $R = R_1 \sqcup R_2$,
and  $ A \in \tau$ if and only if $A= A_1 \sqcup A_2$, where $A_1 \in \tau_1$ and $A_2 \in \tau_2$.  Then there is a canonical isomorphism of
$\dz$-torsors:
$$
[R_1]_{\tau_1}  \otimes_{\dz} [R_2]_{\tau_2}  \lrto [R]_{\tau}
$$
which is given as: $(d_1 \otimes d_2) (A_1 \sqcup A_2) = d_1(A_1) + d_2(A_2)$.

\begin{example}  \em
Let $T$ be a $\dz$-torsor. We consider an element $x \in  T$. We define $\tau$ as a family of all subsets  $A \subset T$ such that
$A \sim (\dz_{\ge 0} + x) $. Then $\tau$ does not depend on the choice of $x \in T$. We obtain a canonical $\dz$-torsor $[T] = [T]_{\tau}$.
There is a canonical isomorphism of $\dz$-torsors:
\begin{gather*}
T \lrto [T] \, \mbox{,} \qquad T \ni y \longmapsto d_y \in [T] \, \mbox{,}
\\
\mbox{where} \; \; d_y(P)= \mid P \setminus (P \cap (\dz_{\ge 0} + y)) \mid
- \mid (\dz_{\ge 0} + y) \setminus       (P \cap (\dz_{\ge 0} + y))          \mid  \; \; \mbox{for any} \; \: P \in \tau \mbox{.}
\end{gather*}
\end{example}

\bigskip

Let $\Gamma$ be a group as in the beginning of this section~\ref{discr_2} (with the fixed central extension~\eqref{centr}).
For any elements $\alpha$ and $\beta$ from $\breve{\Gamma}$
we construct canonically a $\dz$-torsor $[\alpha, \beta]$.
If $\alpha \ge \beta$, then we consider a set
$$
R_{\alpha, \beta} = \{ \gamma \in \Gamma \; : \; \alpha > \gamma \ge \beta  \}
$$
Let $Y \subset R_{\alpha, \beta}$ be a set-theoretic section of $\pi$ over the set $\pi(R_{\alpha, \beta})$.
We recall that for any $n \in \pi(\Gamma)$
the subset $\pi^{-1}(n) \subset \Gamma$ is a $\dz$-torsor. Therefore we can define  a subset ${(\dz_{\ge 0} + Y) \subset \Gamma}$.  Now $\tau$ consists of all subsets $ A \subset R_{\alpha, \beta}$
such that $A \sim (\dz_{\ge 0} + Y) \cap R_{\alpha, \beta}$.
We note that $\tau$ does not depend on the choice of $Y$.
Then we define
$$
[\alpha, \beta] = [R_{\alpha, \beta}] = [R_{\alpha, \beta}]_{\tau}  \, \mbox{.}
$$
If $\alpha \le \beta$, then we define
$$
[\alpha, \beta] = [\beta, \alpha]^*  \, \mbox{.}
$$
(We recall that for a $\dz$-torsor $P$, the $\dz$-torsor $P^* = \Hom_{\dz}(P, \dz)$ has the same set as~$P$, but with the opposite action of $\dz$.)

We have the following canonical isomorphisms, where $\alpha$, $\beta$ and $\gamma$ are from $\breve{\Gamma}$:
\begin{equation}  \label{can}
[\alpha, \alpha] \lrto \dz \quad \mbox{and} \quad [\alpha, \beta] \otimes_{\dz} [\beta, \gamma] \lrto [\alpha, \gamma] \, \mbox{,}
\end{equation}
which satisfy evident compatibility condition and the last isomorphism satisfies the associativity condition for any four elements from $\breve{\Gamma}$.

\bigskip

{\em We fix any $r \in \dc^*$.}

For any two elements $\alpha$ and $\beta$ from $\breve{\Gamma}$
we define a one-dimensional $\dc$-vector space~$\mu_r(\alpha, \beta)$:
$$
\mu_r(\alpha, \beta) = \left([\alpha, \beta] \times^{\dz} \dc^*\right) \cup \{ 0 \}  \, \mbox{,}
$$
where $[\alpha, \beta] \times^{\dz} \dc^*$ is a $\dc^*$-torsor which is obtained from the $\dz$-torsor $[\alpha, \beta]$ and a homomorphism:
$$
\dz \lrto \dc^*  \, \mbox{,} \qquad n \longmapsto r^n  \, \mbox{.}
$$

Evidently, from~\eqref{can} we have  canonical isomorphisms:
\begin{equation}  \label{can_me}
\mu_r(\alpha, \alpha) \lrto \dc \quad \mbox{and} \quad \mu_r(\alpha, \beta) \otimes_{\dc} \mu_r(\beta, \gamma) \lrto
\mu_r(\alpha, \gamma)
\end{equation}
such that they satisfy compatibility condition and the last isomorphism satisfies the associativity condition.

\begin{example} \label{mu_r} \em
We consider $n \in \pi(\Gamma) = \dz$. Then $P_n = \pi^{-1}(n)$ is a $\dz$-torsor.
Besides, ${n =(n, -\infty) \in \breve{\Gamma}}$, and ${n+1 = (n+1, -\infty) \in \breve{\Gamma}}$.
The  one-dimensional $\dc$-vector space $\mu_r(n+1, n)$ is described as the space of all maps $\mu$ from $P_n$ to $\dc$ with the condition
$$
\mu(x +a)=\mu(x) r^{-a} \, \mbox{,} \quad \mbox{where} \quad x \in P_n \, \mbox{,} \quad a \in \dz  \, \mbox{.}
$$
The space $\mu_r(n+1, n)$ is canonically embedded into the space $\D'_+(P_n)$:
\begin{equation}  \label{measures-2}
 \mu   \longmapsto  a_{x,x}= \mu(x)  \, \mbox{,} \quad a_{y,x}= \mu(y) - \mu(y+1)= \mu(y)(1 - r^{-1}) \quad \mbox{for} \quad y < x \in P_n \, \mbox{.}
 \end{equation}

We note also that the natural embedding of the $\dz$-torsor $[n+1,n] = [R_{n+1,n}]$ to the one-dimensional vector space $\mu_r(n+1, n)$ is given as:
\begin{equation}  \label{nmap}
d \longmapsto \mu \, \mbox{,} \quad  \mbox{where} \quad \mu(x) = r^{ d \left(x + \dz_{\ge 0} \right) }  \, \mbox{,} \quad x \in P_n \, \mbox{.}
\end{equation}
\end{example}

\begin{nt} \label{Rm}\em
If a $\dz$-torsor $P_n$ is obtained from a locally compact Abelian group $\oo(n)/\oo(n+1)$ constructed by a two-dimensional local field $K$
with a partial topology and
 with  finite last residue field $\df_q$
(see Proposition~\ref{delta}), then the one-dimensional $\dc$-vector spaces $\mu_q(n+1, n)$ and $\mu(\oo(n)/ \oo(n+1))$
are canonically isomorphic through~\eqref{measures-2}, see also Remark~\ref{deform}.
\end{nt}

\begin{nt}  \label{more-gen} \em
More generally, for elements $\alpha_1$ and $\alpha_2$ from the set $\breve{\Gamma}_K$, where $K $ is a two-dimensional local field
with a partial topology and
 with finite last residue field $\df_q$,  using Remark~\ref{Rm}, we have a canonical isomorphism:
$$
\mu_q(\alpha_1, \alpha_2) \lrto \mu(\alpha_1, \alpha_2)
$$
such that this isomorphism is compatible with the second isomorphism  from  formula~\eqref{can_me}
and with isomorphism~\eqref{measures} for any three elements $\alpha_1, \alpha_2, \alpha_3$ from $\breve{\Gamma}_K$.
\quash{(We use also the following canonical isomorphisms:
$$
\mu_r(\beta_1, \beta_2) \lrto \dc   \qquad \mbox{and} \qquad  \mu(\beta_1, \beta_2) \lrto \dc  \, \mbox{,}
$$
where $\beta_1, \beta_2 \in \breve{\Gamma}_K$ and $\breve{\pi}(\beta_1) = \breve{\pi}(\beta_2)$.)}
\end{nt}

\bigskip

Now we define the spaces of functions and distributions on $\Gamma$:
$$
\D_{+,\alpha,r}(\Gamma)   \qquad \mbox{and} \qquad \D'_{+, \alpha, r}(\Gamma)  \, \mbox{,}
$$
where $\alpha \in \breve{\Gamma}$ and $r \in \dc^*$.

We recall that any integer $m \in \dz$ is also considered as an element $(m, -\infty) \in \breve{\Gamma}$.

\begin{defin} \label{D-space}
The $\dc$-vector space of functions $\D_{+,\alpha,r}(\Gamma) $ is a $\dc$-vector subspace of a space
$$\prod_{k \in \pi(\Gamma)= \dz} \D_+(\pi^{-1}(k))  \otimes_{\dc}
\mu_r(k+1, \alpha)$$
such that an element $\prod\limits_{k \in \pi(\Gamma)} f_k \otimes \mu_k$, where $f_k \in  \D_+(\pi^{-1}(k)) $ and $\mu_k \in \mu_r(k+1, \alpha)$, $\mu_k \ne 0 $, belongs to $\D_{+,\alpha,r}(\Gamma) $
if and only if
$\delta_{+\infty} (f_k) = \eta_{k+1}(f_{k+1})$, where $\mu_{k+1}= \eta_{k+1} \otimes\mu_k$
 and $\eta_{k+1}  \in \mu_r(k+2, k+1)$, for any $ k \in \pi(\Gamma)=\dz $.
\end{defin}

\begin{nt}  \em
This definition can be rewritten as follows.
We recall that there is a map $\breve{\pi} : \breve{\Gamma}  \lrto \dz$ which extends the map $\pi : \Gamma \lrto \dz$.
Let $i : \dz \lrto \breve{\Gamma} = \Gamma \cup \pi(\Gamma) $ be the natural embedding of $\dz = \pi(\Gamma)$.
Then we have
$$
\breve{\pi} \circ i = {\rm id}  \, \mbox{.}
$$
We introduce two linear maps $\breve{\pi}_*$ and $i^*$ from  the $\dc$-vector space
\begin{equation}  \label{space}
{\prod\limits_{k \in \pi(\Gamma)= \dz} \D_+(\pi^{-1}(k))  \otimes_{\dc}
\mu_r(k+1, \alpha) }
\end{equation}
to the $\dc$-vector space ${\prod\limits_{k \in \pi(\Gamma)= \dz} \mu_r(k,\alpha)}$. We consider an element  ${\prod\limits_{k \in \dz} f_k \otimes \mu_k}$, where $f_k \in \D_+(\pi^{-1}(k)) $ and $\mu_k \in \mu_r(k+1, \alpha)$, $\mu_k \ne 0 $. Then we define
$$
\breve{\pi}_*  \left(  \prod_{k \in \dz} f_k \otimes \mu_k  \right) =  \prod_{k \in \dz}  \eta_k(f_k) \cdot \mu_{k-1}  \quad \mbox{and}  \quad
i^* \left( \prod_{k \in \dz} f_k \otimes \mu_k \right) = \prod_{k \in \dz} \delta_{+\infty}(f_{k-1}) \cdot \mu_{k-1}  \, \mbox{,}
$$
where the element $\eta_{k}  \in \mu_r(k+1, k)$
is uniquely defined by an equation $\mu_{k}= \eta_{k} \otimes\mu_{k-1}$. Now an element $w$ from~\eqref{space}
belongs to $\D_{+,\alpha,r}(\Gamma)$
if and only if it satisfies
$$
\breve{\pi}_* (w) = i^* (w)  \, \mbox{.}
$$

\end{nt}

\medskip

\begin{defin}
The $\dc$-vector space of distributions is
$$
\D'_{+,\alpha,r}(\Gamma) =
\left( \bigoplus_{k \in \pi(\Gamma)} \D_+'(\pi^{-1}(k)) \otimes_{\dc}  \mu_r(\alpha, k+1) \right) / Y \, \mbox{,}
$$
where the $\dc$-vector subspace $Y$ is generated by elements
$$
h_k \otimes \mu_k - z_{k+1} \otimes \lambda_{k+1}  \, \mbox{,}
\qquad
 k \in \pi(\Gamma)=\dz \, \mbox{,}  $$
where  $ h_l \mbox{,} \; z_l \in \D'_+(\pi^{-1}(l))$ and $ \mu_l \mbox{,}  \; \lambda_l \in \mu_r(\alpha, l+1)
$,
with the conditions:
$$
h_k = \delta_{+ \infty}  \, \mbox{,} \quad \mu_k = \eta_{k+1} \otimes \lambda_{k+1}  \, \mbox{,}
\quad
 \eta_{k+1} \in \mu_r(k+2, k+1)  \subset \D'_+(\pi^{-1}(k+1)) \, \mbox{,} \quad z_{k+1} = \eta_{k+1} \, \mbox{.}
$$
\end{defin}

We have the following proposition.

\begin{prop} \label{p}
Let $\Gamma$ be a group as in the beginning of section~\ref{discr_2} (with the fixed central extension~\eqref{centr}).
Let $\alpha \in \breve{\Gamma}$, $r \in\dc^*$.
 The natural non-degenerate paring
\begin{equation}  \label{pair1}
\prod_{k \in \pi(\Gamma)= \dz} \D_+(\pi^{-1}(k))  \otimes_{\dc}
\mu_r(k+1, \alpha) \;  \times  \;      \bigoplus_{k \in \pi(\Gamma)} \D_+'(\pi^{-1}(k)) \otimes_{\dc}  \mu_r(\alpha, k+1)               \lrto \dc
\end{equation}
induces a non-degenerate pairing:
\begin{equation}  \label{pair2}
\D_{+, \alpha, r}(\Gamma)  \times \D'_{+, \alpha, r}(\Gamma) \lrto \dc \, \mbox{.}
\end{equation}
\end{prop}
 {  \begin{proof}
For any integer number $m >0$ we consider a $\dc$-vector subspace
$$X_m \subset \prod_{-m \le k <m} \D_+(\pi^{-1}(k))  \otimes_{\dc}
\mu_r(k+1, \alpha)$$
defined by $2m-1$ linear conditions as the corresponding linear conditions in the definition of the space $\D_{+, \alpha, r}(\Gamma)$.
Then it is easy to see that
$$
\D_{+, \alpha, r}(\Gamma) =  \mathop{\lim_{\longleftarrow}}_{m >0 }   X_m    \qquad \mbox{and} \qquad
\D'_{+, \alpha, r}(\Gamma) = \mathop{\lim_{\lrto}}_{m > 0 }   X_m^*  \, \mbox{,}
$$
Besides,  the transition maps in the projective limit are surjective and the transition maps in the inductive limit are injective.
Therefore we obtain a natural non-degenerate pairing.
\end{proof}}

Now we have a theorem.

\begin{Th} \label{Theorem}
Let $K$ be a  two-dimensional local field with a partial topology and  with  finite last residue field $\df_q$.
Let $\alpha \in \breve{\Gamma}_K$.
 Then we have the following properties.
\begin{enumerate}
\item \label{it1} There is a natural surjection of $\dc$-vector spaces  (induced by~\eqref{Phi}):
\begin{equation}  \label{sur}
\D_{\alpha}(K)_{(\oo_K')^*}  \lrto \D_{+,\alpha,q}(\Gamma_K) \, \mbox{.}
\end{equation}
\item  \label{it2} There is a natural injection of $\dc$-vector spaces (induced by~\eqref{Phi'}):
\begin{equation}  \label{inj}
\D'_{+,\alpha,q}(\Gamma_K)  \lrto \D'_{\alpha}(K)^{(\oo_K')^*}   \, \mbox{.}
\end{equation}
\item
The pairing from Proposition~\ref{p} between the $\dc$-vector spaces  $\D_{+,\alpha,q}$ and  $\D'_{+,\alpha,q}$
is induced by  pairing~\eqref{dual_inv_two} between the $\dc$-vector spaces  $\D_{\alpha}(K)_{(\oo_K')^*}$ and $\D'_{\alpha}(K)^{(\oo_K')^*}$.
\end{enumerate}
\end{Th}
{\begin{proof}
We note that for any $k \in \dz$ a $\dz$-torsor $\pi^{-1}(k)$
is canonically identified with the $\dz$-torsor constructed by $1$-dimensional $\bar{K}$-vector space $\oo(k)/ \oo(k+1)$
(see Proposition~\ref{delta}).

As in the proof of Proposition~\ref{prop1}  the spaces $\D_{\alpha}(K)_{(\oo_K')^*}$
and  $\D'_{+,\alpha,q}(\Gamma_K)$ can be defined as ordinary (not double) projective and inductive limits on $m > 0$ of the spaces
$\D(\oo(-m) / \oo(m))_{(\oo_K')^*}  \otimes_{\dc} \mu(m, \alpha)$  and
$\D'(\oo(-m) / \oo(m))^{(\oo_K')^*}  \otimes_{\dc} \mu(\alpha, m)$
  correspondingly.

Now the maps~\eqref{sur}   and~\eqref{inj}   are constructed by means of maps~\eqref{Phi} and~\eqref{Phi'} using also the tensor products with the identity maps on the corresponding $1$-dimensional $\dc$-vector spaces $\mu(\cdot, \cdot)$, and then taking the limits.

To prove that map~\eqref{sur} is surjective we use induction on $m$ as in the proof of Proposition~\ref{W} with the help of Lemma~\ref{lem-int}.

To obtain that the map~\eqref{inj} is injective we note that we take an inductive limit (see~\eqref{Phi'}) which is an exact functor.

The statement on the pairing follows, since pairing~\eqref{pair2} is induced by  pairing~\eqref{pair1} which is induced by  pairing~\eqref{dual_inv_two} in our case.
\end{proof}
}

\begin{example}  \em
We keep notation of Theorem~\ref{Theorem}.
Let $\gamma \in \breve{\Gamma}_K$  and $b \in \mu(\alpha, \gamma)$.
An element $\delta_{\gamma, b}  \in \D'_{\alpha}(K)^{(\oo_K')^*}  $ which was constructed in Section~\ref{charcteristic} belongs to the subspace
  $\D'_{+,\alpha,q}(\Gamma_K)$.
\end{example}

\medskip

Let $\Gamma$  be a group as  in the beginning of section~\ref{discr_2} (with the fixed central extension~\eqref{centr}). We consider a subset $Z \subset \Gamma$
such that for any $n \in \dz$ we have
\begin{equation}  \label{subset}
\pi^{-1}(n)  \cap Z = \dz_{\ge 0} + x
\end{equation}
for some $x \in \pi^{-1}(n)$. We construct an element $\delta_{Z} \in \D_{+, \alpha, r} $, where $\alpha  \in \breve{\Gamma}$
and $r \in \dc^{*}$:
$$
\delta_Z = \prod_{k \in \pi(\Gamma)= \dz} f_{k,Z} \otimes \mu_{k,Z}  \; \subset \; \prod_{k \in \dz}  \D_+(\pi^{-1}(k))  \otimes_{\dc}
\mu_r(k+1, \alpha)
\, \mbox{,}
$$
where the function $f_{k,Z} \in \D(\pi^{-1}(k))$ is equal to $1$ on the subset $\pi^{-1}(k)  \cap Z $ and this function is equal to $0$ on the complement
to this subset, the element ${\mu_{k,Z} \in \mu_r(k+1, \alpha)}$  is uniquely defined by an element
$d_{Z, k+1, \alpha} \in [k+1, \alpha]$
through the canonical embedding  $[\beta_1, \beta_2] \hookrightarrow \mu_r(\beta_1, \beta_2)$ for any $\beta_1, \beta_2 \in \breve{\Gamma}$, and an element
$d_{Z,\beta_1, \beta_2}  \in [\beta_1, \beta_2]$ is defined as:
\begin{gather} \nonumber
d_{Z, \beta_1, \beta_2} \otimes d_{Z, \beta_2, \beta_1} =0 \in \dz \quad \mbox{and}\\  \label{dz}
\mbox{if} \quad
\beta_1 \ge \beta_2 \quad \mbox{then} \quad d_{Z,\beta_1, \beta_2} \in [R_{\beta_1, \beta_2}] \quad \mbox{such that} \quad
d_{Z,\beta_1, \beta_2}  \left(Z \cap R_{\beta_1, \beta_2}\right) =0 \, \mbox{.}
\end{gather}

\begin{nt} \label{E}\em
Let $\Gamma = \Gamma_K$, where $K$ be a two-dimensional local field with  a partial topology. Then an element $\delta_Z   \in \D_{+, \alpha, r} $
is the image  of the characteristic function ${\delta_{E} \in \D_{\alpha}(K)}$
for an appropriate subgroup $E \subset K$ (see section~\ref{charcteristic})
under homomorphism~\eqref{sur}.
\end{nt}

We have a proposition.
\begin{prop} \label{is}
We fix a subset $Z \subset \Gamma$ as above (see~\eqref{subset}). Then an element ${\delta_Z \in  D_{+, \alpha, r}(\Gamma)}$ defines an isomorphism
$\psi_Z$ of $\dc$-vector spaces:
$$
\D_{+, \alpha, r} (\Gamma)  \stackrel{\psi_Z}{\lrto} \Omega  \; \subset \prod_{k \in \dz}  \D_+ (\pi^{-1}(k)) \;   \, \mbox{,} \quad  \mbox{where}
 \quad \psi_Z\left(\prod_{k \in \dz }f_k \otimes \mu_k\right) = \prod_{k \in \dz}(\mu_k/\mu_{k,Z}) f_k \, \mbox{,}
$$
 and the subspace $\Omega$
consists of elements $\prod\limits_{k \in \dz} g_k$
such that
$\delta_{+ \infty}(g_k)= \eta_{k+1, Z}(g_{k+1})   $,
{ where the element $\eta_{k+1,Z}  \in \mu_r(k+2, k+1)$ is
defined from the equality $\mu_{k,Z} \otimes \eta_{k+1, Z} = \mu_{k+1,Z}$ (the element $\eta_{k+1,Z}$ is also the image of an element  $ d_{Z, k+2, k+1} \in [k+2, k+1]$ (see~\eqref{dz}) under the natural map~\eqref{nmap}).}

Similarly, there is an isomorphism $\psi'_Z$ from  $\D'_{+, \alpha, r} (\Gamma)$ to
the quotient space of a space $\bigoplus\limits_{k \in \dz}  \D'_+ (\pi^{-1}(k)) $ by the corresponding subspace.
\end{prop}

\begin{nt} \em
Let $K$ be one of two-dimensional local fields from~\eqref{examples}.
Then the $\dc$-vector space $\D'_{+,\alpha,q}(\Gamma_K) $
is isomorphic to the space $\Psi$ from Definition~1 of~\cite{OsipPar2}
(see also Theorem~1 from~\cite{OsipPar2}). This isomorphism comes from Remark~\ref{E}  and Proposition~\ref{is}
after the fixing a subgroup $E \subset K$ which is isomorphic to one of the following subgroups:
$$
\df_q[[u]]((t)) \, \mbox{,}  \qquad \dz_p((t)) \, \mbox{,}  \qquad \dq_p \cdot (\dz_p[[u]])
$$
(see again~\eqref{examples}).
\end{nt}

\section{Fourier transform}   \label{Sft}

\subsection{Dimension~$1$} \label{dim-1}
We recall the following facts on the Fourier transform in the one-dimensional case.

Let $L$ be a one-dimensional local field with finite  residue field $\df_q$. Then the additive group of $L$ is self-dual with respect to the Pontryagin duality.
This duality is given by a non-degenerate bilinear continuous pairing:
\begin{equation}  \label{ort}
\langle \; \, ,  \; \, \rangle \, : \,  L \times  L      \lrto {\mathbb U}(1)  \, \mbox{,} \qquad  \mbox{where}  \quad {\mathbb U}(1) = \{z \in \dc^* \; \mid  \;   | z | =1      \}  \, \mbox{.}
\end{equation}

For example, if $L \simeq \df_q((u))$, then we fix a non-trivial character $\psi : \df_q \lrto \dc^*$ and choose a non-zero differential form $\omega \in \Omega^1_L$. We define a pairing~\eqref{ort} as
  $$
\langle a, b \rangle_{\omega}  =  \psi (\res (ab \, \omega))    \qquad  \mbox{for} \quad  a \, \mbox{,} \,  b  \, \in L \mbox{.}
$$

Let $\mu \in \mu(L) \subset \D'(L)$ be a Haar measure, $\mu \ne 0$. Then we have a Fourier transform:
\begin{gather*}
\F_{\mu ,\omega} \; : \; \D(L)  \lrto \D(L)  \\
 \F_{\mu ,\omega}(f)(y)= \int_L f(x) \langle -y, x \rangle_{\omega} \, d \mu(x) \, \mbox{,}
\qquad \mbox{where} \quad f \in \D(L) \quad \mbox{and} \quad y \in L  \, \mbox{.}
\end{gather*}
Thus, the Fourier transform depends on the choices of a character $\psi$, a differential form $\omega $ and a measure $\mu$. We omit the character
$\psi$ in the notation for the Fourier transform, since this character will be fixed throughout the whole paper.

Multiplicative group $L^*$ of the field $L$ acts on the additive group of the field $L$ by multiplication:
$$
\mbox{for any} \quad l \in L^* \quad \mbox{we have an isomorphism}   \quad  L \stackrel{\times l}{\longrightarrow} L \, \mbox{,} \quad x \longmapsto lx \, \mbox{,}
\quad \mbox{where}  \quad x \in L \, \mbox{.}
$$
This isomorphism  defines an action
$l^*$ on the $\dc$-vector  spaces $\D(L)$, $\D'(L)$ and $\mu(L)$, where $l^*(f)(x)= f(l^{-1} x)$ and so on (for $x \in L$, $f \in \D(L)$).

We have the following properties
\begin{equation} \label{F1}
\F_{\mu ,\omega}(l^*(f))= (\mu/ l^*(\mu)) \, (l^{-1})^*(\F_{\mu ,\omega}(f))  \, \mbox{,}
\end{equation}
$$
\F_{\mu ,l\omega}(f)= (l^{-1})^* \F_{\mu ,\omega}(f)   \, \mbox{.}
$$

 We consider an element $\mu^{-1} \in \mu(L)$ which is  is uniquely defined by the property $\mu^{-1}(E^{\perp})= \mu(E)^{-1}$, where $E \subset L$ is an open compact subgroup and the ortogonal subgroup $E^{\perp}$ is taken with respect to the pairing~\eqref{ort}.
From~\eqref{F1}, for $G \in \D'(L)$ and for a map
$$\F_{\mu^{-1} ,\omega} \; :  \; \D'(L)  \lrto \D'(L) $$
 (which is conjugate to the map $\F_{\mu^{-1} ,\omega} : \D(L)  \lrto \D(L)$ with respect to~\eqref{dual_1}) we have
\begin{equation}  \label{F2}
\F_{\mu^{-1} ,\omega}(l^*(G))= (l^*(\mu)/ \mu) \, (l^{-1})^* (\F_{\mu^{-1} ,\omega}(G))  \, \mbox{.}
\end{equation}

Since the group $\oo_L^*$ acts trivially on the space $\mu(L)$, from~\eqref{F1} and~\eqref{F2} we obtain that the maps $\F_{\mu ,\omega}$ and $\F_{\mu^{-1} ,\omega}$ induce the well-defined  maps   on the $\dc$-vector  spaces $\D(L)_{\oo_L^*}$ and $\D'(L)^{\oo_L^*}$.
For the characteristic function $\delta_{m_L^n}$ of a fractional ideal $m_L^n$, where $n \in \dz$, we have
$$
\F_{\mu ,\omega}(\delta_{m_L^n}) = \mu(m_L^n) \, \delta_{(m_L^n)^{\perp}}  \, \mbox{,}
$$
where $\mu(m_L^n) = q^{-n} \mu(\oo_L)$ and   $\delta_{(m_L^n)^{\perp}} = \delta_{m_L^{-n - \nu_L(w)}}$  (when $L \simeq \df_q((u))$).

\subsection{Rank (or dimension)~$2$}

\subsubsection{Involution on the set $\breve{\Gamma}$}   \label{involu}
Let $\Gamma$ be a group as in the beginning of section~\ref{discr_2} (with the fixed central extension~\eqref{centr}).
The group $\Gamma$ acts on the set $\breve{\Gamma}$ such that this action extends the action which comes from the group operation in $\Gamma$ and the formula for the remaining action is:
$$
(n,p) \circ (m, -\infty) = (n+m , -\infty)  \, \mbox{.}
$$
Hence we obtain a commutative operation
$$
\breve{\Gamma} \times \breve{\Gamma} \lrto \breve{\Gamma}  \, \mbox{,} \qquad \alpha \times \beta  \longmapsto \alpha \circ \beta
$$
such that this operation extends the action of $\Gamma$ on $\breve{\Gamma}$ and it is equal to the addition in the group $\dz$ when it is restricted to $\pi(\Gamma)$.

\medskip

We fix an element $\gamma \in \Gamma$.

We {\em define} an element  $\beta^{\perp(\gamma)} \in \breve{\Gamma}$
for any element  $\beta \in \breve{\Gamma}$ by means of the following easy proposition.

\begin{prop}
For any element $\beta \in \breve{\Gamma}$ there exists  a well-defined element $\beta^{\perp(\gamma)} \in \breve{\Gamma}$
which is the minimum among all elements $\varphi \in \breve{\Gamma}$
with the property:
$$
\varphi \circ \beta \circ \gamma  >0  \, \mbox{,}
$$
where $0 \in \Gamma \subset \breve{\Gamma}$ is the zero of the group $\Gamma$.
\end{prop}
It is clear that $\beta \longmapsto \beta^{\perp(\gamma)}$ is an involution on the set $\breve{\Gamma}$, and this involution preserves the subset $\Gamma$.

\begin{example}
Let $\Gamma = \dz \oplus \dz$, where the first group $\dz$  corresponds to $\pi(\Gamma)$.  Then we have
$$
(n, -\infty)^{\perp((0 \oplus 0))} = (-n+1, - \infty)  \qquad \mbox{and} \qquad
(a \oplus b)^{\perp((0 \oplus 0))} = (-a  \oplus -b+1)  \, \mbox{,}
$$
where $n,a,b $ are from $\dz$.
\end{example}

\begin{prop} \label{inol-prop}
We have the following properties.
\begin{enumerate}
\item For any $\gamma_1, \gamma_2 \in \Gamma$ and for any $\beta \in \breve{\Gamma}$ we have
$$
\beta^{\perp(\gamma_1 + \gamma_2)} =  (- \gamma_2) \circ \beta^{\perp(\gamma_1)}   \, \mbox{.}
$$
\item
\label{item-2}
For any $\phi, \gamma \in \Gamma$ and for any $\alpha \in \breve{\Gamma}$ we have
$$
(\phi \circ \alpha)^{\perp(\gamma)} = (- \phi)  \circ \alpha^{\perp(\gamma)}  \, \mbox{.}
 $$
 \item
 \label{item-3}
  For any $\gamma \in \Gamma$ and $\alpha, \beta \in \breve{\Gamma}$
 there is a canonical isomorphism of $\dz$-torsors
 \begin{equation}  \label{ab}
 [\alpha, \beta]  \lrto [\alpha^{\perp(\gamma)}, \beta^{\perp(\gamma)}]
 \end{equation}
 which is compatible with isomorphisms~\eqref{can} for any three elements from $\breve{\Gamma}$.
 \item For any $\phi, \gamma \in \Gamma$ and $\alpha, \beta \in \breve{\Gamma}$ there is a commutative diagram
 \begin{equation}  \label{diagram}
 \xymatrix{
[\alpha, \beta ]  \ar[r]  \ar[d] & [\alpha^{\perp(\gamma)}, \beta^{\perp(\gamma)}] \ar[d]\\
[\phi \circ \alpha, \phi \circ \beta ]  \ar[r] & [(- \phi)  \circ \alpha^{\perp(\gamma)}, (- \phi)  \circ \beta^{\perp(\gamma)}]
}
 \end{equation}
 where the horizontal arrows are isomorphisms which follow from items~\ref{item-2} and~\ref{item-3}, and the vertical arrows are isomorphisms induced by actions of elements $\phi$ and $- \phi$ correspondingly.
\end{enumerate}
\end{prop}
\begin{proof}
The first two statements follow from the definition of the involution on $\breve{\Gamma}$.

For the proof of the third statement  it is enough to suppose that $\alpha \ge \beta$ and then to construct a canonical isomorphism:
\begin{equation}  \label{R-can}
[R_{\alpha, \beta}] \lrto  [R_{\beta^{\perp(\gamma)}, \alpha^{\perp(\gamma)}}]^*  \, \mbox{.}
\end{equation}
Indeed, then we would have isomorphisms (for $\alpha \ge \beta$) :
$$
[\alpha, \beta] = [R_{\alpha, \beta}] \lrto  [R_{\beta^{\perp(\gamma)}, \alpha^{\perp(\gamma)}}]^* = [\alpha^{\perp(\gamma)}, \beta^{\perp(\gamma)}] \, \mbox{.}
$$

We note that for a $\dz$-torsor $P$ the $\dz$-torsor $P^*$ is canonically isomorphic to the set of all functions $f$ from $P$ to $\dz$
with the property $f(x +a)= f(x) +a$, where $x \in P$, $a \in \dz$, and the action of $\dz$ is the following: $(b +f )(x) = f(x) +b$ for $b \in \dz$.

For any set  $A$  from $\tau(R_{\alpha, \beta})$ (see the definition of $R_{\alpha, \beta}$ in section~\ref{discr_2})  we define a set $A^{\perp, \gamma}$ from  $\tau(R_{\beta^{\perp(\gamma)}, \alpha^{\perp(\gamma)}})$ in the following way:
$$
A^{\perp, \gamma} = \{  \varphi \in \Gamma \; : \; \gamma + \varphi + \phi  \ne 0  \quad \mbox{for any} \quad \phi \in A                               \} \, \mbox{.}
$$

We consider $d_1 \in [R_{\alpha, \beta}]$.
We construct  now $f \in [R_{\beta^{\perp(\gamma)}, \alpha^{\perp(\gamma)}}]^*$    by $d_1$.
For
\linebreak
 ${d_2  \in [R_{\beta^{\perp(\gamma)}, \alpha^{\perp(\gamma)}}]}$  we define
$$
f(d_2)  = d_1(A)  + d_2 (A^{\perp(\gamma)})  \in \dz  \, \mbox{,}
$$
where $A$ is a set from $\tau(R_{\alpha, \beta})$. The integer $f(d_2)$ does not depend on the choice of the set $A$.
The map $d_1 \longmapsto f$ gives isomorphism~\eqref{R-can}, and hence it gives isomorphism~\eqref{ab}, which is compatible with isomorphisms~\eqref{can} for any three elements from $\breve{\Gamma}$.

The fourth item follows from the second item and an explicit description of the isomorphism in the third item.
\end{proof}

\begin{cons} \label{cons-me}
For any $r \in \dc^*$ and  any $\gamma \in \Gamma$,  $\alpha, \beta \in \breve{\Gamma}$ there is a canonical isomorphism of one-dimensional
$\dc$-vector spaces:
$$
\mu_r(\alpha, \beta)  \lrto \mu_r(\alpha^{\perp(\gamma)}, \beta^{\perp(\gamma)})
$$
 which is compatible with isomorphisms~\eqref{can_me} for any three elements from $\breve{\Gamma}$. Moreover, there is a commutative diagram which is analogous to diagram~\ref{diagram} when we change $\dz$-torsors $[ \cdot  , \cdot  ]  $ to corresponding one-dimensional $\dc$-vector spaces $\mu_r(\cdot, \cdot)$.
\end{cons}

\subsubsection{Fourier transform and two-dimensional local fields}  \label{FT2}

Let $K$ be a two-dimensional local field which is isomorphic to $\df_q((u))((t))$.

We consider a non-zero differential $2$-form $\omega \in \Omega^2_K$.
Then $\omega = g \frac{du}{u} \wedge \frac{dt}{t}$, where $g \in K$. We denote
$$
\ord_K(\omega) = \nu_K(g)  \,  \in   \, \Gamma_K  \, \mbox{.}
$$
We note that $\ord_K(\omega)$ does not depend on the choice of isomorphism $K \simeq \df_q((u))((t))$. Indeed, if we choose another local parameters
$u'$ and $t'$ of the two-dimensional local field $K$, then it is easy to see that
$$
\frac{du}{u} \wedge \frac{dt}{t} = h \frac{du'}{u'} \wedge \frac{dt'}{t'}  \, \mbox{,}
$$
where $h \in (\oo_K')^*$.

We consider a non-degenerate pairing:
\begin{equation}  \label{two-pairing}
\langle  \; \, ,  \; \,   \rangle_{\omega}   \; :  \;
  K \times K \lrto {\mathbb U}(1)  \, \mbox{,}  \quad \langle a, b \rangle_{\omega}  = \psi (\res\nolimits_K (ab \, w)) \, \mbox{,} \quad a,b \in K \, \mbox{,}
\end{equation}
where $\psi : \df_q \lrto \dc^*$ is a fixed non-trivial character, and the residue
$$
\res\nolimits_K \left(\sum_{i,j} a_{i,j} u^it^j du \wedge dt \right) = a_{-1,-1}  \, \in \df_q  \, \mbox{.}
$$
The pairing~\eqref{two-pairing} does not depend on the choice of local parameters $u, t$, since the residue $\res_K$ does not depend on this choice.
(We omit the character $\psi$ here and later in the definition of the two-dimensional Fourier tranforms for~$K$, since we have fixed~$\psi$, see also
section~\ref{dim-1}.)

We note that for any $ \alpha_1 \le \alpha_2 \in \breve{\Gamma}_K$ the pairing~\eqref{two-pairing} induces the Pontryagin duality between locally compact Abelian groups
$$
\oo'(\alpha_1) / \oo'(\alpha_2) \qquad  \mbox{and}  \qquad \oo'(\alpha_2^{\perp(\ord(\omega))})/ \oo'(\alpha_1^{\perp(\ord(\omega))}) \, \mbox{.}
$$
Therefore the Fourier transform gives a linear  map:
$$
 \D(\oo'(\alpha_1) / \oo'(\alpha_2) )  \otimes_{\dc} \mu(\alpha_2, \alpha_1)  \lrto  \D(\oo'(\alpha_2^{\perp(\ord(\omega))}) / \oo'(\alpha_1^{\perp(\ord(\omega))}))  \, \mbox{.}
$$
Hence, from~\eqref{measures}, from Corollary~\ref{cons-me} of Proposition~\ref{inol-prop} and from Remark~\ref{more-gen}   we obtain the composition of canonical isomorphisms for any $\alpha \in \breve{\Gamma}$:
\begin{multline*}
 \D(\oo'(\alpha_1) / \oo'(\alpha_2) )  \otimes_{\dc} \mu(\alpha_2, \alpha) \lrto
\D(\oo'(\alpha_2^{\perp(\ord(\omega))}) / \oo'(\alpha_1^{\perp(\ord(\omega))})) \otimes_{\dc} \mu(\alpha_1, \alpha)  \lrto \\ \lrto
\D(\oo'(\alpha_2^{\perp(\ord(\omega))}) / \oo'(\alpha_1^{\perp(\ord(\omega))})) \otimes_{\dc} \mu(\alpha_1^{\perp(\ord(\omega))}, \alpha^{\perp(\ord(\omega))})  \mbox{.}
\end{multline*}
These maps  are compatible for various $\alpha_1$ and $\alpha_2$ with transition maps in projective limits of formula~\eqref{D}, see more details and proofs in~\cite[\S~5.4]{OsipPar1}. Therefore these maps induce a well-defined two-dimensional Fourier transform for any $\alpha \in \breve{\Gamma}$:
\begin{equation}  \label{FD}
{\bf F}_{\omega}  \; : \; \D_{\alpha}(K)  \lrto \D_{\alpha^{\perp(\ord(\omega))} }(K)  \, \mbox{.}
\end{equation}
On the dual side, we obtain a well-defined  two-dimensional Fourier transform on the space of distributions (we denote this map by the same letter as in the case of the space of functions):
\begin{equation}  \label{FD'}
{\bf F}_{\omega}  \; : \; \D'_{\alpha}(K)  \lrto \D'_{\alpha^{\perp(\ord(\omega))} }(K)  \, \mbox{.}
\end{equation}
The Fourier transforms
$$
{\bf F}_{\omega}  \; : \; \D_{\alpha}(K)  \lrto \D_{\alpha^{\perp(\ord(\omega))} }(K) \qquad \mbox{and} \qquad
{\bf F}_{\omega}  \; : \; \D'_{\alpha^{\perp(\ord(\omega))}}(K)  \lrto \D'_{\alpha }(K)
$$
 are conjugate to each other with respect to the pairing~\eqref{pairing-D}.

From~\cite[Prop.~25]{OsipPar1} we obtain that the maps ${\bf F}_{\omega}$ induce well-defined linear maps for any $\alpha \in \breve{\Gamma}$ (again denoted by the same letter):
\begin{equation}  \label{FO}
{\bf F}_{\omega}  \; : \; \D_{\alpha}(K)_{(\oo'_K)^*}  \lrto \D_{\alpha^{\perp(\ord(\omega))} }(K)_{(\oo'_K)^*}
 \mbox{,} \quad
{\bf F}_{\omega}  \; : \; \D_{\alpha}(K)^{(\oo'_K)^*}  \lrto \D_{\alpha^{\perp(\ord(\omega))} }(K)^{(\oo'_K)^*}    \mbox{.}
\end{equation}
such that the maps
$$
{\bf F}_{\omega}  \; : \; \D_{\alpha}(K)_{(\oo'_K)^*}  \lrto \D_{\alpha^{\perp(\ord(\omega))} }(K)_{(\oo'_K)^*}  \quad \mbox{and}
\quad
{\bf F}_{\omega}  \; : \; \D_{\alpha^{\perp(\ord(\omega))}}(K)^{(\oo'_K)^*}  \lrto \D_{\alpha }(K)^{(\oo'_K)^*}
$$
 are conjugate to each other with respect to the pairing~\eqref{dual_inv_two}.

From the properties of Fourier transforms~\eqref{FD}-\eqref{FD'} (see~\cite[Prop.~24]{OsipPar1})
we have the following property for the both maps from~\eqref{FO}:
$$
{\bf F}_{\omega} \circ {\bf F}_{\omega} = {\rm id}  \, \mbox{.}
$$
(Here the second map ${\bf F}_{\omega}$ is applied to the space of functions or distributions on $K$, and this space depends on the index $\alpha^{\perp(\ord(\omega))}$.)

\begin{example} \em
For any $\alpha \in \breve{\Gamma}_K$, any $\beta \in \breve{\Gamma}_K$, any $b \in \mu(\alpha, \beta)$
we have constructed in section~\ref{charcteristic} a characteristic function $\delta_{\beta,b} \in \D'_{\alpha}(K)$
of a fractional ideal $\oo'(\beta)$.

Now, using the definition of the Fourier transform ${\bf F}_{\omega}$, we obtain
$$
{\bf F}_{\omega} (\delta_{\beta,b}) = \delta_{\beta^{\perp(\ord(\omega))}  ,b}  \, \mbox{,}
$$
where $\delta_{\beta^{\perp(\ord(\omega))}  ,b}  \in \D_{\alpha^{\perp(\ord(\omega) )}}$  is the characteristic function
of the fractional ideal  $\oo'(\beta^{\perp(\ord{\omega)}})$,
and we used the canonical isomorphism from Corollary~\ref{cons-me} of Proposition~\ref{inol-prop}:
$$
\mu(\alpha, \beta)  \lrto \mu(\alpha^{\perp(\ord(\omega))},  \beta^{\perp(\ord(\omega))})  \, \mbox{.}
$$
(Therefore we consider the element $b$ also as an element  from the space  \linebreak $\mu(\alpha^{\perp(\ord(\omega))},  \beta^{\perp(\ord(\omega))})$.)
\end{example}

\subsubsection{Fourier transform and free Abelian groups of rank~$2$}
In this section we define Fourier transform for the spaces of functions and distributions on a group $\Gamma$, which were constructed
in section~\ref{discr_2}, and relate them with  the Fourier transform ${\bf F}_{\omega}$  when $\Gamma = \Gamma_K$, where $K \simeq \df_q((u))((t))$.

Let $\Gamma$ be a group as in the beginning of section~\ref{discr_2} (with the fixed central extension~\eqref{centr}).

We fix $r \in \dc^*$ and $\gamma \in \Gamma$.

For any $k \in \dz $ we consider $\eta_k \in \mu_r(k+1, k) \subset \D_+'(\pi^{-1}(k))$ (see formula~\eqref{measures-2}).
We note   that
\begin{equation}  \label{inv}
(k, -\infty)^{\perp(\gamma)} = (-k - \pi(\gamma)+1, -\infty) \mbox{.}
\end{equation}

We define a linear map between $\dc$-vector spaces:
\begin{gather}  \label{Ft}
{\bf F}_{\eta_k} \; : \; \D_+(\pi^{-1}(k))  \lrto \D_+( \pi^{-1}(-k - \pi(\gamma)) )  \, \mbox{,} \\  \nonumber
{\bf F}_{\eta_k}(\delta_{\ge x})= \eta_k(x) \, \delta_{\ge x^{\perp(\gamma)}} \, \mbox{,}  \quad x \in \pi^{-1}(k)  \mbox{.}
\end{gather}
Here $\eta_k(x)=\eta_k(\delta_{\ge x}) $. The elements $\delta_{\ge x}$, where $x \in \pi^{-1}(k)$,
are the basis of the {$\dc$-vector} space $\D_+(\pi^{-1}(k))$.
Besides, $ x^{\perp(\gamma)}  \in \pi^{-1}(-k - \pi(\gamma))$.

As the dual map to the map~\eqref{Ft} we consider the map (which we will denote by the same letter):
\begin{equation}  \label{Fp}
{\bf F}_{\eta_k} \; : \; \D'_+( \pi^{-1}(-k - \pi(\gamma)) ) \lrto \D'_+(\pi^{-1}(k)) \, \mbox{.}
\end{equation}
Explicitly, this map is given as the following:
\begin{equation}
\label{Fp'}
  {\bf F}_{\eta_k} (\{ a_{x,x}\}) = \{ \tilde{a}_{y,y}\}  \, \mbox{,} \quad \mbox{where} \quad
  x \in  \pi^{-1}(-k - \pi(\gamma)) \, \mbox{,} \; \, y \in \pi^{-1}(k) \, \mbox{,} \;
\,   \tilde{a}_{y,y}= \eta_k(y) \,  a_{y^{\perp(\gamma)}, y^{\perp(\gamma)}} \,  \mbox{.}
\end{equation}

\medskip

We suppose that $\eta_k  \ne 0$. By the isomorphism from Corollary~\ref{cons-me} of Proposition~\ref{inol-prop}  we have also that $$\eta_k \in \mu_r(-k -\pi(\gamma), -k -\pi(\gamma) +1) \, \mbox{.}$$
We introduce an element
\begin{gather}
\eta_k^{-1}  \in \mu_r(-k -\pi(\gamma) +1, -k -\pi(\gamma) )  \label{minus} \\ \nonumber \mbox{uniquely defined by the property} \quad \eta_k \otimes \eta_{k}^{-1} = 1  \in \dc = \mu_r(-k -\pi(\gamma), -k - \pi(\gamma))  \, \mbox{.}
\end{gather}

From construction of  isomorphism~\eqref{ab}, and hence of isomorphism from Corollary~\ref{cons-me} of Proposition~\ref{inol-prop}, we have for both maps from formulas~\eqref{Ft}-\eqref{Fp}:
\begin{equation}  \label{ident}
{\bf F}_{\eta_k^{-1}} \circ {\bf F}_{\eta_k} = {\rm id}  \, \mbox{.}
\end{equation}

\begin{lemma} \label{lem2}
We have the following properties:
$$
{\bf F}_{\eta_k}(\eta_k^{-1})= \delta_{+ \infty}   \qquad \mbox{and}  \qquad
{\bf F}_{\eta_k}(\delta_{+ \infty})= \eta_k  \, \mbox{.}
$$
\end{lemma}
\begin{proof}
From~\eqref{ident} we have that it is enough to prove only one property: the other will follow. We will prove the second property.
We recall that $\delta_{+ \infty}= \{a_{x,x} \}$ with $a_{x,x}=1$ for any $x \in \pi^{-1}(-k - \pi(\gamma))$. Then from~\eqref{Fp'} we have
$$
{\bf F}_{\eta_k} (\delta_{+ \infty}) = \{\tilde{a}_{y,y}\}  \, \mbox{,} \quad \mbox{where}  \quad
\tilde{a}_{y,y} = \eta_k(y)  \quad \mbox{for any}  \quad y \in \pi^{-1}(k)  \, \mbox{.}
$$
Thus, we obtained ${\bf F}_{\eta_k}(\delta_{+ \infty})= \eta_k$.
\end{proof}

\medskip

Now we give the {\em definitions} of Fourier transforms ${\bf F}_{\gamma}$ for the spaces of functions and distributions on  $\Gamma$
by means of the following proposition (see formulas~\eqref{Ffl}-\eqref{Fdl} below).

\begin{prop} \label{defin-prop-Fourier}
Let $r \in \dc^*$, $\alpha \in \breve{\Gamma}$ and $\gamma \in \Gamma$. We have the following maps.
\begin{enumerate}
\item
A linear map from the $\dc$-vector space $\prod\limits_{k \in \dz} \D_+(\pi^{-1}(k)) \otimes_{\dc} \mu_r(k+1, \alpha)$
to the $\dc$-vector space $\prod\limits_{k \in \dz} \D_+(\pi^{-1}(k)) \otimes_{\dc} \mu_r(k+1, \alpha^{\perp(\gamma)})$ given as
\begin{equation}  \label{Ftr}
\prod_{k \in \dz} f_k \otimes \mu_k  \longmapsto \prod_{k \in \dz} {\bf F}_{\eta_k}(f_k) \otimes \mu_{k-1}  \, \mbox{,}
\end{equation}
where $f_k \in \D_+(\pi^{-1}(k))$, $0 \ne \mu_k \in  \mu_r(k+1, \alpha) $,  $\mu_k = \mu_{k-1} \otimes \eta_k$,
${\eta_k \in \mu_r(k+1, k)}$, and also $\mu_{k-1} \in \mu_r(k, \alpha) = \mu_r (-k -\pi(\gamma) +1, \alpha^{\perp(\gamma)}) $ (see~\eqref{inv} and~Corollary~\ref{cons-me} of Proposition~\ref{inol-prop}),
induces a well-defined linear map (Fourier transform) between $\dc$-vector subspaces:
\begin{equation} \label{Ffl}
{\bf F}_{\gamma} \; : \; \D_{+, \alpha, r}(\Gamma)  \lrto \D_{+, \alpha^{\perp(\gamma)}, r}(\Gamma)  \, \mbox{.}
\end{equation}
\item
A  linear map from the $\dc$-vector space
$\bigoplus\limits_{k \in \dz} \D'_+(\pi^{-1}(k)) \otimes_{\dc} \mu_r(\alpha, k+1)$
to the $\dc$-vector space
$
\bigoplus\limits_{k \in \dz} \D'_+(\pi^{-1}(k)) \otimes_{\dc} \mu_r(\alpha^{\perp(\gamma)}, k+1)
$
given as
$$
\bigoplus_{k \in \dz} f_{-k - \pi(\gamma)} \otimes \mu_{-k - \pi(\gamma)}
\longmapsto \bigoplus_{k \in \dz}
{\bf F}_{\eta_k} (f_{-k - \pi(\gamma)}) \otimes \mu_{-k - \pi(\gamma) -1}  \, \mbox{,}
$$
where $f_{-k - \pi(\gamma)}  \in  \D'_+(\pi^{-1}(-k - \pi(\gamma))) $,
\begin{gather*}
0 \ne \mu_{-k - \pi(\gamma)}  \in \mu_r(\alpha, -k - \pi(\gamma) +1) = \mu_r(\alpha^{\perp(\gamma)}, k) \, \mbox{,}\\
0 \ne \mu_{-k - \pi(\gamma) -1}  \in \mu_r(\alpha, -k - \pi(\gamma)) = \mu_r(\alpha^{\perp(\gamma)}, k+1) \, \mbox{,}
\end{gather*}
$\eta_k \in \mu_r(k+1, k)$
such that $\mu_{-k - \pi(\gamma) -1} \otimes \eta_k = \mu_{-k - \pi(\gamma)}$,
induces a well-defined linear map (Fourier transform) between $\dc$-vector quotient spaces:
\begin{equation}  \label{Fdl}
{\bf F}_{\gamma} \; : \; \D'_{+, \alpha, r}(\Gamma)  \lrto \D'_{+, \alpha^{\perp(\gamma)}, r}(\Gamma)  \, \mbox{.}
\end{equation}
\end{enumerate}
\end{prop}
\begin{proof}
We prove the first statement (the second is proved similarly).
We note that after the application of map~\eqref{Ftr} ''the neighbour elements'' in the infinite product will be
${\bf F}_{\eta_{k+1}}(f_{k+1}) \otimes \mu_k$ and ${\bf F}_{\eta_k}(f_k) \otimes \mu_{k-1}$. Therefore, according to definition~\ref{D-space}  we have to check that the condition $\eta_{k+1}(f_{k+1})= \delta_{+\infty}(f_{k})$ for any $k \in \dz$ implies the condition
$$
\eta_k^{-1}( {\bf F}_{\eta_k}(f_k) )  = \delta_{+ \infty}({\bf F}_{\eta_{k+1}}(f_{k+1}))  \, \mbox{,}
$$
where we consider $\eta_k^{-1}  \in \mu_r(-k -\pi(\gamma) +1, -k -\pi(\gamma) )$ (see~\eqref{minus}).
We have
\begin{gather*}
\eta_k^{-1}( {\bf F}_{\eta_k}(f_k) ) =  {\bf F}_{\eta_k}(\eta_k^{-1})(f_k) = \delta_{+ \infty}(f_k)  \, \mbox{,}\\
 \delta_{+ \infty}({\bf F}_{\eta_{k+1}}(f_{k+1})) = {\bf F}_{\eta_{k+1}} (\delta_{+ \infty})(f_{k+1}) = \eta_{k+1}(f_{k+1}) = \delta_{+\infty}(f_k)
 \, \mbox{,}
\end{gather*}
where we used Lemma~\ref{lem2}.
\end{proof}

\medskip

We have a theorem.

\begin{Th}Let $r \in \dc^*$, $\alpha \in \breve{\Gamma}$ and $\gamma \in \Gamma$. We have the following properties.
\begin{enumerate}
\item The Fourier tranforms
$$
{\bf F}_{\gamma} \; : \; \D_{+, \alpha, r}(\Gamma)  \lrto \D_{+, \alpha^{\perp(\gamma)}, r}(\Gamma) \qquad \mbox{and} \qquad
{\bf F}_{\gamma} \; : \; \D'_{+, \alpha^{\perp(\gamma)}, r}(\Gamma)  \lrto \D'_{+, \alpha, r}(\Gamma)
$$
are conjugate to each other with respect to the pairing from Proposition~\ref{p}.
\item For  both Fourier tranforms~\eqref{Ffl} and~\eqref{Fdl}
we have
$${\bf F}_{\gamma} \circ {\bf F}_{\gamma} = {\rm id}  \, \mbox{.} $$
(Here the second map ${\bf F}_{\gamma}$ is applied to the space of functions or distributions on $\Gamma$ which depends on the index $\alpha^{\perp(\gamma)}$.)
\item Let $\Gamma = \Gamma_K$, where $K \simeq \df_q((u))((t))$. Let $0 \ne \omega  \in \Omega^2_{K}$,  $\gamma = \ord_K(\omega)$.
The following diagrams are commutative:
$$
\xymatrix{
\D_{\alpha}(K)_{(\oo_K')^*}  \ar[r]^(0.46){{\bf F}_{\omega}}
\ar@{->>}[d] &
\D_{\alpha^{\perp(\gamma)}}(K)_{(\oo_K')^*} \ar@{->>}[d] \\
\D_{+,\alpha, q}(\Gamma_K)   \ar[r]^{{\bf F}_{\gamma}}
 &
\D_{+,\alpha^{\perp(\gamma)}, q}(\Gamma_K)
}
\qquad
\qquad
\xymatrix{
\D'_{+,\alpha, q}(\Gamma_K)  \ar[r]^{{\bf F}_{\gamma}}
\ar@{^{(}->}[d] &
\D'_{+,\alpha^{\perp(\gamma)}, q}(\Gamma_K) \ar@{^{(}->}[d] \\
\D'_{\alpha}(K)^{(\oo_K')^*}
\ar[r]^(0.46){{\bf F}_{\omega}} &
\D'_{\alpha^{\perp(\gamma)}}(K)^{(\oo_K')^*}
}
$$
where the vertical arrows are the natural maps from items~\ref{it1}  and~\ref{it2} of Theorem~\ref{Theorem}.
\end{enumerate}
\end{Th}
\begin{proof}
The first item follows from the constructions of the Fourier transforms ${\bf F}_{\gamma}$ (see also~\eqref{Ft}-\eqref{Fp}) and the pairing from Proposition~\ref{p}.

The second item follows from the construction of the Fourier transforms ${\bf F}_{\gamma}$ and the property~\eqref{ident}.

The third item follows from the definition of the Fourier transforms ${\bf F}_{\omega}$ given in section~\ref{FT2}, the construction of maps from
items~\ref{it1}  and~\ref{it2} of Theorem~\ref{Theorem}, the commutative diagram
\begin{equation}  \label{comm-diagr}
\xymatrix{
\D(\oo(k)/ \oo(k+1))_{(\oo_K')^*}  \ar[r]^(.39){{\bf F}_{\eta_k}}
\ar[d] &
\D(\oo(-k - \pi(\gamma))/ \oo(-k -\pi(\gamma)+1))_{(\oo_K')^*}
\ar[d]
\\
\D_{+}(\pi^{-1}(k))   \ar[r]^{{\bf F}_{\eta_k}}
 &
\D_{+}(\pi^{-1}(-k - \pi(\gamma)))
}
\end{equation}
where $k \in \dz$, $\eta_k \in \mu_q(k+1, k) = \mu(\oo(k)/ \oo(k+1))$, the horizontal arrows are the corresponding Fourier transforms,  and all the arrows are isomorphisms, and the commutative diagram which consists of dual maps to the maps in diagram~\eqref{comm-diagr}
$$
\xymatrix{
\D'(\oo(k)/ \oo(k+1))^{(\oo_K')^*}
 &
\D'(\oo(-k - \pi(\gamma))/ \oo(-k -\pi(\gamma)+1))^{(\oo_K')^*} \ar[l]_(0.59){{\bf F}_{\eta_k}}
\\
\D'_{+}(\pi^{-1}(k))  \ar[u]
 &
\D'_{+}(\pi^{-1}(-k - \pi(\gamma)))   \ar[u]  \ar[l]_{{\bf F}_{\eta_k}}
}
$$
\end{proof}

\section{Action of the discrete Heisenberg group}  \label{Hei}

\subsection{Discrete Heisenberg group}
 \subsubsection{Various definitions of the discrete Heisenberg group}

Let $\Gamma$ be a group as in the beginning of section~\ref{discr_2} (with the fixed central extension~\eqref{centr}).

From the set $\breve{\Gamma}$ we obtain a connected groupoid, where the objects of this groupoid  are elements of $\breve{\Gamma}$, and for any
$\alpha_1, \alpha_2 \in \breve{\Gamma} $ the set of morphisms between $\alpha_1$ and $\alpha_2$ is a $\dz$-torsor $[\alpha_1, \alpha_2]$. The group $\Gamma$ acts on the set $\breve{\Gamma}$ (see section~\ref{involu}) and on this groupoid in the natural way. Therefore for any $\alpha \in \breve{\Gamma}$ we have a central extension of  $\Gamma$ (see, e.g,~\cite{Br}):
\begin{equation}  \label{centr_ext}
0 \lrto \dz \lrto \tilde{\Gamma}_{\alpha}  \lrto \Gamma  \lrto 0  \, \mbox{.}
\end{equation}
The concrete description of the group $\tilde{\Gamma}_{\alpha} $
is the following:
$$
\tilde{\Gamma}_{\alpha} = \left\{  (\beta, h)  \mid \beta \in \Gamma \, \mbox{,} \; h \in [\alpha, \beta \circ \alpha]    \right\}  \, \mbox{.}
$$
And the group law in $\tilde{\Gamma}_{\alpha} $ is given as:
$$
(\beta_1, h_1) (\beta_2, h_2) = (\beta_1 + \beta_2, h_1 \otimes \beta_1(h_2))  \, \mbox{,}
$$
where   $\beta_1(h_2) \in [\beta_1 \circ \alpha, (\beta_1 + \beta_2) \circ \alpha]$ comes from the natural isomorphism of $\dz$-torsors:
  $$
  [\alpha, \beta_2 \circ \alpha]  \lrto [\beta_1 \circ \alpha, (\beta_1 + \beta_2) \circ \alpha] \quad \mbox{given by the action of} \quad \beta_1 \, \mbox{,}
  $$
  and
  multiplication $\otimes$ is given by~\eqref{can}.

\begin{defin}
The discrete Heisenberg group ${\rm Heis}(3, \dz)$
is the group of matrices of the form
$$
\begin{pmatrix}
1 & a & c \\
0 & 1 & b \\
0 & 0 & 1
\end{pmatrix}
$$
where $a,b,c \in \dz$.
\end{defin}

We can think on ${\rm Heis}(3, \dz)$ as the group of triples $(a,b,c)$, where $a,b,c \in  \dz$, with the group law
$$
(a_1, b_1, c_1)(a_2, b_2, c_2)= (a_1 + a_2, b_1 + b_2, c_1 + c_2 + a_1 b_2)
$$

The group ${\rm Heis}(3, \dz)$ has two generators $x = (1,0,0)$ and $y =(0,1,0)$ and  relations
$$
z = xyx^{-1}y^{-1} \, \mbox{,} \quad  xz=zx \, \mbox{,} \quad yz=zy  \, \mbox{,}
$$
where $z = (0,0,1)$ is the generator of the center of ${\rm Heis}(3, \dz)$.
The discrete Heisenberg group is a non-abelian nilpotent group of nilpotence class $2$.

\medskip

We recall that for any $g \in [\alpha_1, \alpha_2]$, where $\alpha_1, \alpha_2 \in \breve{\Gamma}$, we have a uniquely defined element
$g^{-1} \in [\alpha_2, \alpha_1]$ such that $g \otimes g^{-1} =0 \in \dz$.

\begin{prop}  \label{prop-Heis}
We have the following properties.
\begin{enumerate}
\item  \label{item1} For any $\alpha_1, \alpha_2 \in \breve{\Gamma}$ there is a canonical isomorphism of groups (which is also an isomorphism of the corresponding central extensions):
$$\Lambda_{\alpha_2, \alpha_1} \; \, : \; \, \tilde{\Gamma}_{\alpha_1}   \lrto \tilde{\Gamma}_{\alpha_2} \quad \mbox{given as} \quad (\beta, h)   \longmapsto (\beta, g \otimes h \otimes \beta(g^{-1}))  \, \mbox{,}$$
where $g \in [\alpha_2, \alpha_1]$. This isomorphism does not depend on the choice of $g$.
\item
\label{iso-Heis}
The fixing of section of the homomorphism $\pi$ from~\eqref{centr}  gives the isomorphism ${\Gamma \simeq \dz \oplus \dz}$ (where the first group $\dz$  corresponds to $\pi(\Gamma)$), which defines an isomorphism
$$
{\rm Heis}(3, \dz)  \lrto
\tilde{\Gamma}_{(0, -\infty)} \quad \mbox{given as} \quad
(a,b,c) \longmapsto (( b \oplus a ), -c + d_{b})  \, \mbox{,}
$$
where $d_b \in [(0, -\infty), (b, -\infty) ]$ is uniquely defined by the following conditions :
\begin{gather*}
d_b \left( \bigcup_{b \le n \le -1}  \left(n \oplus \dz_{\ge 0} \right) \right) =0  \quad  \mbox{if}  \quad     b < 0 \, \mbox{,}\\
\mbox{and} \qquad d_b \otimes (d_{-b}) = 0  \quad  \mbox{if}  \quad  b   \ge 0  \quad \mbox{(compare with~\eqref{dz})} \mbox{.}
\end{gather*}
\end{enumerate}
\end{prop}
\begin{proof}
The first item follows directly from definitions of groups~$\tilde{\Gamma}_{\alpha_1}$ and~$\tilde{\Gamma}_{\alpha_2}$.

For the proof of the second item we note that for any $a_1, b_1, a_2, b_2, c_1, c_2 \in \dz$ we have
$$
(0 \oplus a_1) (d_{b_2}) = -a_1 b_2 + d_{b_2}  \, \mbox{,}
$$
and hence
\begin{gather*}
d_{b_1} \otimes (b_1 \oplus a_1)(d_{b_2})= -a_1 b_2 + d_{b_1 + b_2}  \, \mbox{,}\\
( -c_1 + d_{b_1}) \otimes (b_1 \oplus a_1)( -c_2 + d_{b_2})= -c_1 -c_2  -a_1 b_2 + d_{b_1 + b_2}
\, \mbox{.}
\end{gather*}
Using the last formula, we apply direct computations with group laws  for elements $(a_1,b_1,c_1)$ and $(a_2, b_2, c_2)$ from ${\rm Heis}(3, \dz)$.
\end{proof}

\begin{nt} \em
The isomorphism constructed in item~\ref{iso-Heis} of Proposition~\ref{prop-Heis} gives the following commutative diagram:
$$
\xymatrix{
{\rm Heis}(3, \dz)
  \ar[rr]
\ar[d] & & \tilde{\Gamma}_{(0, -\infty)}
\ar[d]
\\
\dz \oplus \dz  \ar[rr]^{(a  \oplus b)  \longmapsto (b \oplus a)}
 & &
\dz \oplus \dz
}
$$
where the vertical arrows are the natural maps which are the quotient maps by the center of a group (we have also fixed a section of the homomorphism $\pi$ and, thus, we fixed an isomorphism $\Gamma \simeq \dz \oplus \dz$.)

The flip $(a \oplus b)  \longmapsto (b \oplus a)$
in the map which is the bottom arrow of the diagram can be changed to the identity map (and then there is an isomorphism given by  the upper arrow which gives  isomorphism of corresponding central extensions) if we change all $\dz$-torsors $[\alpha_1, \alpha_2]$, where $\alpha_1, \alpha_2 \in \breve{\Gamma}$, to the dual $\dz$-torsors, because a central extension~\eqref{centr_ext}  will change under the last change to a central extension given by the opposite (i.e. taken with the sign minus) $2$-cocycle.

Indeed,  by the universal coefficient theorem (compare also with~\cite{Bre}), for any Abelian groups $A$ and $B$ we have an exact sequence
$$
0 \lrto \mathop{{\rm Ext}}\nolimits^1(B,A)  \lrto H^2(B,A)  \stackrel{\theta}{\lrto} \Hom (\bigwedge\nolimits_{\dz}^2 B , A )  \lrto 0  \, \mbox{,}
$$
where we consider $A$ as a trivial $B$-module, and $\theta(f)(x,y)= f(x,y) - f(y,x)$ for a $2$-cocycle $f$ which gives the element from  $H^2(B,A) $, and $x,y \in B$. For $B \simeq \dz \oplus \dz$ and $A \simeq \dz$ we have $\mathop{{\rm Ext}}\nolimits^1(B,A) =0 $. Therefore, in this case,  an element from $H^2(B,A)$ is uniquely defined by its image under the map $\theta$.
We have $\bigwedge_{\dz}^2 B \simeq \dz$ and an automorphism of  $\bigwedge_{\dz}^2 B$ induced by
 the permutation of groups in $B \simeq \dz \oplus \dz$ is
 the multiplication by $-1$. Therefore the permutation of groups in $B \simeq \dz \oplus \dz$ is
  the same as the multiplication to $-1$ of an element from $H^2(B,A)$.

We note also that for a $2$-cocycle $f$ which gives   central extension~\eqref{centr_ext}, we have
\begin{equation}  \label{det}
\theta(f)((n_1 \oplus p_1) \wedge (n_2 \oplus p_2)) = p_1 n_2 - n_1 p_2  \, \mbox{,}\qquad \mbox{where} \quad n_1, p_1, n_2, p_2 \in \dz \, \mbox{,}
\end{equation}
 and we used a section of the homomorphism $\pi$ which gives the isomorphism $\Gamma \simeq \dz \oplus \dz$. But the result in formula~\eqref{det} does not depend on the choice of a section of $\pi$.
\end{nt}

\begin{nt} \em
The isomorphism constructed in item~\ref{iso-Heis} of Proposition~\ref{prop-Heis} depends on the choice of a section of the homomorphism $\pi$.
It will be written in more general situation in Remark~\ref{section}  later, how this isomorphism changes after another choice of a section of the homomorphism $\pi$.

\end{nt}

\subsubsection{Action of the discrete Heisenberg group on function and distribution spaces}
\label{action}

Now we define an action of the group $\tilde{\Gamma}_{\alpha}$ on the $\dc$-vector spaces $\D_{+, \alpha, r}(\Gamma)$ and
$\D'_{+, \alpha, r}(\Gamma)$, where $\Gamma$ is a group as in the beginning of section~\ref{discr_2}, $\alpha \in \breve{\Gamma}$
and $r \in \dc^*$.

We recall that for any elements $\alpha_1, \alpha_2$ from the set $\breve{\Gamma}$ we have the natural embedding of sets
$$
[\alpha_1, \alpha_2]  \hookrightarrow \mu_r(\alpha_1, \alpha_2) \qquad \quad \mbox{given as} \qquad h \longmapsto h \times 1 \, \mbox{.}
$$
We denote this map by $\vartheta_r$.

We consider an element $(\beta, h) \in \tilde{\Gamma}_{\alpha}$, where $\beta \in \Gamma$ and $h \in [\alpha, \beta \circ \alpha]$.
We consider also an element
$$
Q = \prod_{k \in \dz} f_k \otimes \mu_k  \, \mbox{,} \qquad \mbox{where} \qquad f_k \in \D_+(\pi^{-1}(k)) \, \mbox{,} \quad 0 \ne \mu_k \in
\mu_r(k+1, \alpha) \, \mbox{.}
$$
We define
$$
(\beta, h) \circ Q = \prod_{k \in \dz} \beta(f_k) \otimes \left( \beta(\mu_k) \otimes \vartheta_r(h^{-1}) \right) \, \mbox{,}
$$
where $\beta(f_k) \in  \D_+(\pi^{-1}(k + \pi(\beta)))$ comes from the natural isomorphism induced by the action of the element $\beta$,
and $\beta(\mu_k) \in \mu_r( k+1+ \pi(\beta) , \beta \circ \alpha )$. Hence we have that
$$
\beta(\mu_k) \otimes \vartheta_r(h^{-1})  \in \mu_r(k+1 + \pi(\beta), \alpha)  \, \mbox{.}
$$

It is easy to see that this action of the element $(\beta, h)$ on the $\dc$-vector space $\prod\limits_{k \in \dz} \D_+(\pi^{-1}(k)) \otimes_{\dc} \mu_r(k+1, \alpha)$  preserves the $\dc$-vector subspace $\D_{+, \alpha, r}(\Gamma)$ (see Definition~\ref{D-space}). Thus we obtain the action of the group $\tilde{\Gamma}_{\alpha}$ on the $\dc$-vector space $\D_{+, \alpha, r}(\Gamma)$.

Similarly, we have an action of the group $\tilde{\Gamma}_{\alpha}$
on the $\dc$-vector space $\D'_{+, \alpha, r}(\Gamma)$. This action is induced by the following action:
$$
(\beta, h)  \circ  \bigoplus_{k \in \dz}   \left(g_k  \otimes \lambda_k \right)    =
\bigoplus_{k \in \dz}  \left( \beta(g_k)  \otimes \left( \vartheta_r(h) \otimes  \beta(\lambda_k)        \right)    \right)   \, \mbox{,}
$$
where $\beta \in \Gamma$, $h \in [\alpha, \beta \circ \alpha]$, $g_k \in \D'_+(\pi^{-1}(k))$ and $\lambda_k \in \mu_r(\alpha, k+1)$.

 From the described constructions  we have the following formula for any $Q \in \D_{+, \alpha, r}(\Gamma)$, $ S \in \D'_{+, \alpha, r}(\Gamma) $, $\varpi \in \tilde{\Gamma}_{\alpha}$:
\begin{equation} \label{inv_pair}
\langle Q, S  \rangle  = \langle  \varpi \circ Q, \varpi \circ S  \rangle   \, \mbox{,}
\end{equation}
where $\langle \cdot, \cdot \rangle$  is pairing~\eqref{pair2}.

\bigskip

Let $K$ be a two-dimensional local field with  a partial topology and with  finite last residue field $\df_q$. Let $\alpha \in \breve{\Gamma}_K$.
We have the group $
\tilde{\Gamma}_{\alpha}$ for $\Gamma=\Gamma_K$. In this case we will denote the group $\tilde{\Gamma}_{\alpha}$ by $\tilde{\Gamma}_{K, \alpha}$.
 We construct now an action of the group
$\tilde{\Gamma}_{K,\alpha}$ on the $\dc$-vector spaces $\D_{\alpha}(K)_{(\oo'_K)^*}$  and $\D'_{\alpha}(K)^{(\oo'_K)^*}$. We connect also these actions with actions of the group $\tilde{\Gamma}_{\alpha}$ described above via the natural maps from Theorem~\ref{Theorem}  (when $r = q$ and $\Gamma = \Gamma_K$).

We recall that we have the rank-$2$ valuation  $\nu_K : K^*  \lrto \Gamma_K$.
We also recall that for any $\alpha_1, \alpha_2 \in \Gamma_K$ there is a natural isomorphism $\mu_q(\alpha_1, \alpha_2) \simeq \mu(\alpha_1, \alpha_2)$ which is compatible with
isomorphisms  from~\eqref{can_me}
and~\eqref{measures}, and also with
  the action of the group~$\Gamma_K$.

For any integers $m \ge n $ let $f_{n,m} \otimes \mu_m$ be an element from projective system which gives the element from  the space  $\D_{\alpha}(K)_{(\oo'_K)^*}$,
see Proposition~\ref{prop1}. (Here ${f_{n,m}  \in \D(\oo(n)/ \oo(m))_{(\oo'_K)^*}}$ and $\mu_m  \in \mu(m, \alpha)$.) For any
$(\beta, h) \in \tilde{\Gamma}_{K, \alpha}$ we define an action
$$
(\beta, h) \circ \left( f_{n,m} \otimes \mu_m    \right) = \beta'(f_{n,m}) \otimes \left( \beta'( \mu_m) \otimes \vartheta_q(h^{-1})  \right) \, \mbox{,}
$$
where $\beta' \in K^*$ such that $\nu_K(\beta')= \beta$, and
$$\beta'(f_{n,m}) \in \D(\oo(n + \pi(\beta))/ \oo(m + \pi(\beta)))_{(\oo'_K)^*} \, \mbox{,} \qquad
\beta'( \mu_m) \otimes \vartheta_q(h^{-1}) \in \mu(m + \pi(\beta), \alpha) \, \mbox{.}$$
This action does not depend on the choice of $\beta'$ (we recall that the group $(\oo'_K)^*$ acts trivially on the space $\mu(\alpha_1, \alpha_2)$
for any $\alpha_1, \alpha_2 \in \breve{\Gamma}_K$) and defines the  well-defined action of the group $\tilde{\Gamma}_{K,\alpha}$
on the $\dc$-vector space $\D_{\alpha}(K)_{(\oo'_K)^*}$.

Similarly, we define an action of the group $\tilde{\Gamma}_{K,\alpha}$
on the $\dc$-vector space $\D'_{\alpha}(K)^{(\oo'_K)^*}$
such that the natural pairing~\eqref{dual_inv_two}  is invariant with respect to the action of the group $\tilde{\Gamma}_{K,\alpha}$ (compare also with formula~\eqref{inv_pair}).

From above constructions we obtain a proposition.

\begin{prop}  \label{prop-equiv}
Let $K$ be a two-dimensional local field with  a partial topology and  with  finite last residue field $\df_q$. Let  $\alpha \in \breve{\Gamma}_K$.
 The natural maps
 $$
 \xymatrix{
\D_{\alpha}(K)_{(\oo_K')^*} \, \ar@{->>}[r] & \, \D_{+,\alpha,q}(\Gamma_K})
\qquad \mbox{and} \qquad
\xymatrix{
\D'_{+,\alpha,q}(\Gamma_K) \, \ar@{^{(}->}[r]  & \, \D'_{\alpha}(K)^{(\oo_K')^*}
}
 $$
   (see Theorem~\ref{Theorem}) are $\tilde{\Gamma}_{K,\alpha}$-equivariant.
\end{prop}

\subsection{Extended discrete Heisenberg group}

In this section we define and describe an extended discrete Heisenberg group such that the discrete Heisenberg group is a subgroup of this group.
We also define an action of the extended discrete Heisenberg group on the spaces of functions and distributions such that this action extends the action of the discrete Heisenberg group.

\subsubsection{Definition and action of the extended discrete Heisenberg group}  \label{act-ext-heis}

We recall that we consider a central extension:
$$
0 \lrto \dz \lrto \Gamma \stackrel{\pi}{\lrto} \dz \lrto 0  \, \mbox{.}
$$

We consider a subgroup $C$ of the automorphism group of the group $\Gamma$ such that this subgroup preserves this central extension and acts identically on both groups $\dz$ from this central extension.

This central extension can be described as the set of $\dz$-torsors $\pi^{-1}(n)$, where $n \in \dz$, such that there are  isomorphisms of $\dz$-torsors:
$$
\pi^{-1}(n_1)  \otimes_{\dz} \pi^{-1}(n_2)  \lrto \pi^{-1}(n_1 + n_2)  \qquad \mbox{for any}  \quad n_1, n_2 \in \dz \, \mbox{,}
$$
and these isomorphisms satisfy the associativity conditions for any $n_1 , n_2 , n_3 \in \dz$.
Hence the group $C$ can be identified with the automorphism group of the $\dz$-torsor $\pi^{-1}(1)$. Therefore $C = \dz$. Explicitly, we have (compare also with~\eqref{ch}):
$$
k \diamond (n,p)= (n, kn + p )  \, \mbox{,} \quad \mbox{where}  \quad n \in \pi(\Gamma) \, \mbox{,}  \quad p \in \pi^{-1}(n) \,  \mbox{,} \quad k \in C = \dz  \, \mbox{.}
$$
We extend this action on the set $\breve{\Gamma}$:
$$
k \diamond (n, -\infty) = (n, - \infty)  \, \mbox{.}
$$

Thus the group $\dz=C$ acts on the set $\breve{\Gamma}$. Moreover, we have
$$
k \diamond (\alpha_1 \circ \alpha_2) = (k \diamond \alpha_1 ) \circ (k \diamond \alpha_2)   \quad
\mbox{for}  \quad k \in \dz =C \, \mbox{,} \quad \alpha_1 \, \mbox{,}  \alpha_2 \in \breve{\Gamma} \, \mbox{.}
$$

From the construction of the $\dz$-torsor $[\alpha_1, \alpha_2]$, where $\alpha_1, \alpha_2 \in \breve{\Gamma}$,
it is easy to see that the action of the group $\dz = C$ on the set $\breve{\Gamma}$
is naturally extended to an action of the group $\dz = C$ on the groupoid associated with $\breve{\Gamma}$.
(We recall that this groupoid  is a category, where the objects are elements of the set $\breve{\Gamma}$ and the set of morphisms between $\alpha_1$ and $\alpha_2$ from $\breve{\Gamma}$ is the $\dz$-torsor $[\alpha_1, \alpha_2]$.)

Moreover, we have also a property:
$$
k \diamond (\beta(h))  = (k \diamond \beta)(k(h))   \quad \mbox{for} \quad k \in C = \dz \, \mbox{,}  \quad \beta \in \Gamma \, \mbox{,}
\quad h \in [\alpha_1, \alpha_2] \, \mbox{,} \quad \alpha_1, \alpha_2 \in \breve{\Gamma} \, \mbox{,}
$$
and $\beta( \cdot)$, $k(\cdot)$ and $(k \diamond \beta) (\cdot)$
 are the results of the action of elements $\beta \in \Gamma$, $k \in C$ and $k \diamond \beta \in \Gamma$
 on corresponding elements of corresponding $\dz$-torsors.

From these properties and from the construction of the group $\tilde{\Gamma}_{\alpha}$, where $\alpha \in \breve{\Gamma}$,
we obtain the following  isomorphism of groups for any $m \in C =\dz$ and $\alpha \in \breve{\Gamma}$:
$$
T_{m, \alpha}  \; : \;   \tilde{\Gamma}_{\alpha}  \lrto  \tilde{\Gamma}_{m \diamond \alpha}  \, \mbox{,} \qquad
T_{m, \alpha} ((\beta, h)) = (m \diamond \beta, m(h) ) \, \mbox{,}
$$
where $(\beta, h) \in \tilde{\Gamma}_{\alpha}$.

Besides, we have for any $m,k \in C=\dz$ that the composition of $T_{k, \alpha}$ with $T_{m, k \diamond \alpha }$ is equal to $T_{m + k,  \alpha }$.

We recall that for any $\alpha_1, \alpha_2 \in \breve{\Gamma}$ we have constructed in item~\ref{item1}  of Proposition~\ref{prop-Heis} the canonical isomorphism $\Lambda_{\alpha_2, \alpha_1}$
from the group $\tilde{\Gamma}_{\alpha_1}$ to the group   $\tilde{\Gamma}_{\alpha_2}$.
From constructions we obtain at once that the following diagram is commutative:
\begin{equation}  \label{square}
 \xymatrix{
\tilde{\Gamma}_{\alpha_1}
  \ar[rr]^{T_{m, \alpha_1}}
\ar[d]_{\Lambda_{\alpha_2, \alpha_1}} & & \tilde{\Gamma}_{m \diamond \alpha_1}
\ar[d]^{\Lambda_{m \diamond \alpha_2, m \diamond \alpha_1}}
\\
\tilde{\Gamma}_{\alpha_2}  \ar[rr]_{T_{m, \alpha_2}}
 & &
 \tilde{\Gamma}_{m \diamond \alpha_2}
}
\end{equation}

Now for any $\alpha \in \breve{\Gamma}$ an element $m \in C = \dz$ defines an automorphism of the group $\tilde{\Gamma}_{\alpha}$ which is equal to the composition of $T_{m, \alpha}$ with $\Lambda_{ \alpha, m \diamond \alpha}$. From diagram~\eqref{square}   we obtain that these automorphisms define the homomorphism (which is a monomorphism) from the group
$C = \dz$  to the group $\mathop{\rm Aut} \tilde{\Gamma}_{\alpha}$.

\begin{defin}
Let $\alpha \in \breve{\Gamma}$.
We define an extended discrete Heisenberg group
$$
\hat{\Gamma}_{\alpha} = \tilde{\Gamma}_{\alpha}  \rtimes \dz  \, \mbox{,}
$$
where the homomorphism $\dz = C \lrto \mathop{\rm Aut} \tilde{\Gamma}_{\alpha}$ is just described above.
\end{defin}

Explicitly, we have the group law in~$\hat{\Gamma}_{\alpha}$:
$$
(f_1, m_1) (f_2,m_2)= (f_1 \, m_1(f_2), m_1 + m_2 )  \, \mbox{,}
$$
where $f_1, f_2 \in \tilde{\Gamma}_{\alpha}$, $m_1, m_2 \in \dz=C$ and $m_1(f_2)$ is the image of $f_2$ under the automorphism of $\tilde{\Gamma}_{\alpha}$ given by $m_1$.

From diagram~\eqref{square}   it is easy to obtain a proposition.

\begin{prop} \label{map}
Let $\alpha_1, \alpha_2 \in \breve{\Gamma}$. Then the map
$$
\hat{\Gamma}_{\alpha_1} \, \ni \, (f,m)  \longmapsto (\Lambda_{\alpha_2, \alpha_1}(f), m) \, \in \, \hat{\Gamma}_{\alpha_2} \, \mbox{,} \qquad  \mbox{where}
\quad
f \in \tilde{\Gamma}_{\alpha} \, \mbox{,}  \quad m \in \dz =C  \, \mbox{,}
$$
gives a   canonical isomorphism from the group $\hat{\Gamma}_{\alpha_1}$
to the group $\hat{\Gamma}_{\alpha_2}$.
\end{prop}

\bigskip

Now we define  actions of the group  $\hat{\Gamma}_{\alpha}$, where $\alpha \in \breve{\Gamma}$, on the $\dc$-vector spaces $\D_{+, \alpha, r}(\Gamma)$ and
$\D'_{+, \alpha, r}(\Gamma)$  such that these actions extend the actions
of the subgroup $\tilde{\Gamma}_{\alpha}$ described in section~\ref{action}.

For this goal we have to define an action of the group $\dz = C$  on the $\dc$-vector spaces $\D_{+, \alpha, r}(\Gamma)$ and
$\D'_{+, \alpha, r}(\Gamma)$.
First, we define
these actions when $\alpha= \alpha_0=(0, -\infty)$.

We define
$$
m \diamond \prod_{k \in \dz} (f_k \otimes \mu_k) = \prod_{k \in \dz} \left( m(f_k) \otimes m(\mu_k) \right)  \, \mbox{,}
$$
where $m \in \dz=C$, $f_k \in \D_+(\pi^{-1}(k))$, $0 \ne \mu_k \in \mu_r(k+1, 0)$, and $m( \cdot)$ are the results of  corresponding actions of an element $m$
on elements from $\D_+(\pi^{-1}(k))$  or  from $\mu_r(k+1, 0)$. It is easy to see that this action defines the action of the group $\dz=C$
on the $\dc$-vector space  $\prod\limits_{k \in \dz} \D_+(\pi^{-1}(k)) \otimes_{\dc} \mu_r(k+1, 0)$ such that this action preserves the subspace $\D_{+, \alpha_0, r}(\Gamma)$.

Similarly, we define an action of the group $\dz = C$ on the $\dc$-vector space $\D'_{+, \alpha_0, r}(\Gamma)$ as the well-defined action induced by the following action:
$$
m \diamond \bigoplus_{k \in \dz} (g_k \otimes \lambda_k) = \bigoplus_{k \in \dz} \left( m(g_k) \otimes m(\lambda_k)     \right)   \, \mbox{,}
 $$
 where  $m \in \dz=C$, $g_k \in \D'_+(\pi^{-1}(k))$ and $\lambda_k \in \mu_r(0, k+1)$.

Now we  define the action for arbitrary $\alpha \in \breve{\Gamma}$.
We choose an element $h \in \mu_r(\alpha, \alpha_0)$ such that $h \ne 0$. Then for any $m \in \dz=C$ and any $Q \in \D_{+, \alpha, r}(\Gamma)$,
 $S \in \D'_{+, \alpha, r}(\Gamma)$
 we define
$$
m \diamond Q = (m \diamond (Q \otimes h)) \otimes h^{-1}   \, \mbox{,}
\qquad
m \diamond S = (m \diamond  (S \otimes h^{-1}) \otimes h  \, \mbox{,}
$$
where we used that $Q \otimes h \in \D_{+, \alpha_0, r}(\Gamma)$ and $S \otimes h^{-1}  \in \D'_{+, \alpha_0, r}(\Gamma)$. Clearly, the last definition does not depend on the choice of $h$.

At last we write  actions of an element $(f,m) \in \hat{\Gamma}_{\alpha}$:
\begin{equation}\label{actions-ex}
(f,m) \diamond Q = f \circ ( m \diamond Q )  \, \mbox{,} \qquad (f,m) \diamond S = f \circ (m \diamond S)  \, \mbox{,}
\end{equation}
where  $Q \in \D_{+, \alpha, r}(\Gamma)$  and
 $S \in \D'_{+, \alpha, r}(\Gamma)$.  These formulas will  define actions of the group $\hat{\Gamma}_{\alpha}$ if  we prove that
 $$
 m(f) \diamond Q = m \diamond (f \circ (-m \diamond (Q)))   \quad \mbox{and} \quad  m(f) \diamond S = m \diamond (f \circ (-m \diamond (S))) \, \mbox{,}
 $$
 where $f \in \tilde{\Gamma}_{\alpha}$, $m \in \dz=C$. The last formulas follow from a lemma.

 \begin{lemma} \label{well-def}
 Let $\alpha \in \breve{\Gamma}$, $m \in \dz=C$, $f \in \tilde{\Gamma}_{\alpha}$, $Q \in \D_{+, \alpha, r}(\Gamma)$  and
 $S \in \D'_{+, \alpha, r}(\Gamma)$. We have
 $$
 m \diamond (f \circ Q) = m(f) \circ (m \diamond Q)    \qquad \mbox{and}  \qquad  m \diamond (f \circ S) = m(f) \circ (m \diamond S)  \, \mbox{.}
 $$
 \end{lemma}
 \begin{proof}
 We prove the first formula, the proof of the second formula is similar. By Proposition~\ref{map}   and  the following formula (which is easy to check)
 $$
 g \circ R = h^{-1} \otimes  ( \Lambda_{\alpha_0, \alpha}(g) \circ ( R \otimes h)  )  \, \mbox{,}
 $$
 where $g \in \tilde{\Gamma}_{\alpha}$,  $R \in \D_{+, \alpha, r}(\Gamma)$, $\alpha_0 \in \breve{\Gamma}$, $0 \ne h \in \mu_r(\alpha, \alpha_0)$,
 we reduce the proof to the case when $\alpha = \alpha_0=(0, -\infty)$. In this case the proof is easy and follows at once, by constructions of actions, because
 $m \diamond (0, -\infty) = (0 , -\infty)$.
 \end{proof}

\medskip

We note that it is easy to see that the natural non-degenerate pairing between $\dc$-vector spaces  $\D_{+, \alpha, r}(\Gamma)$ and $\D'_{+, \alpha, r}(\Gamma)$ (see Proposition~\ref{p}) is invariant with respect to the action of the group $\hat{\Gamma}_{\alpha}$ on these spaces.

\bigskip

\begin{nt}  \em
The action of the group $\dz=C$ on the group $\tilde{\Gamma}_{\alpha}$, where $\alpha \in \breve{\Gamma}$,
is an analog of the ``loop rotation'' in the theory of loop groups, see~\cite[ch.~7.6, ch.~8.6, ch.~14.1]{PS}.
The loop rotation from~\cite{PS} can be presented as transformations $t \mapsto u t$, where $u \in \mathbb{U}(1) \subset \dc^*$,
of  a local parameter in a field $\dc((t))$. This transformation   gives also an automorphism of the field $\dc((t))$. The field $\dc((t))$ is an Archimedean analog of a two-dimensional local field $L((t))$, where $L$
is a one-dimensional local field with  finite  residue field $\df_q$, i.e. $L$ is either isomorphic to $\df_q((t))$ or isomorphic to
a finite extension of  the field $\dq_p$.
\end{nt}

\begin{nt}  \em
We consider $K = L((t))$, where $L$
is a one-dimensional local field with finite residue field $\df_q$. We fix a local parameter $u \in \oo_L$. We consider $\alpha \in \breve{\Gamma}_K$.
Using the local parameter $u$,
we describe an action of the group ${\hat{\Gamma}}_{K, \alpha}= \hat{\Gamma}_{\alpha}$, where $\Gamma = \Gamma_K$, on the $\dc$-vector spaces
 $\D_{\alpha}(K)_{(\oo'_K)^*}$  and $\D'_{\alpha}(K)^{(\oo'_K)^*}$ (compare with~\cite[\S~5]{OsipPar2}, where it was done for $\alpha = \alpha_0 = (0, -\infty)$ and  the space $\D'_{\alpha_0}(K)^{(\oo'_K)^*}$).

 The group $\dz$ acts on the field $K$ (by automorphisms):
 $$
 \dz \ni m  \longmapsto \left\{ \sum_{n \in \dz} z_n t^n  \longmapsto     \sum_{n \in \dz} z_n (u^mt)^n        \, \mbox{,} \quad z_n \in L            \right\} \, \mbox{.}
 $$
 This action induces the actions of the group $\dz$ on the $\dc$-vector spaces $\D( \oo(n)/ \oo(k) )_{(\oo'_K)^*}$,
 $ \D'( \oo(n)/ \oo(k) )^{(\oo'_K)^*} $ for any $k \ge n \in \dz $ and on  $\mu(k,n)$ for any $k,n \in \dz= \pi(\Gamma_K)$.

 Moreover, these actions
 give the well-define actions of the group $\dz$ on the $\dc$-vector spaces $\D_{\alpha_0}(K)_{(\oo'_K)^*}$  and $\D'_{\alpha_0}(K)^{(\oo'_K)^*}$
via formulas given by inductive and projective limits  from Proposition~\ref{prop1}. Therefore we obtain  actions of the group $\dz $ on the spaces $\D_{\alpha}(K)_{(\oo'_K)^*}$  and $\D'_{\alpha}(K)^{(\oo'_K)^*}$:
$$
m \diamond Q = (m \diamond (Q \otimes h)) \otimes h^{-1}   \, \mbox{,}
\qquad
m \diamond S = (m \diamond  (S \otimes h^{-1}) \otimes h  \, \mbox{,}
$$
 where  $h \in \mu(\alpha, \alpha_0)$ such that $h \ne 0$, $m \in \dz$,  $Q \in \D_{\alpha}(K)_{(\oo'_K)^*} $,
 $S \in \D'_{\alpha}(K)^{(\oo'_K)^*}$ (the actions does not depend on the choice of $h$).

Hence we define actions of  an element $(f,m) \in {\hat{\Gamma}}_{K, \alpha}$:
$$
(f,m) \diamond Q = f \circ (m \diamond Q)   \, \mbox{,}  \qquad (f,m) \diamond S = f \circ (m \diamond S)  \, \mbox{,}
$$
where the actions of the element  $f \in  {\tilde{\Gamma}}_{K, \alpha}$ were described in Section~\ref{action} (before Proposition~\ref{prop-equiv}). This gives the well-defined actions of the group ${\hat{\Gamma}}_{K, \alpha}$ on the $\dc$-vector spaces $\D_{\alpha_0}(K)_{(\oo'_K)^*}$  and $\D'_{\alpha_0}(K)^{(\oo'_K)^*}$. (We have to use an analog of Lemma~\ref{well-def}  to obtain that these actions are  well-defined.)

Besides, by constructions of actions, we obtain that  the natural pairing~\eqref{dual_inv_two}
   between the $\dc$-vector spaces $\D_{\alpha}(K)_{(\oo'_K)^*}$ and $\D'_{\alpha_0}(K)^{(\oo'_K)^*}$
   is invariant with respect to the action of the group ${\hat{\Gamma}}_{K, \alpha}$,
and the natural maps
 $$
 \xymatrix{
\D_{\alpha}(K)_{(\oo_K')^*} \, \ar@{->>}[r] & \, \D_{+,\alpha,q}(\Gamma_K})
\qquad \mbox{and} \qquad
\xymatrix{
\D'_{+,\alpha,q}(\Gamma_K) \, \ar@{^{(}->}[r]  & \, \D'_{\alpha}(K)^{(\oo_K')^*}
}
 $$
  are $\hat{\Gamma}_{K,\alpha}$-equivariant.
\end{nt}

\subsubsection{Explicit description of the  extended discrete Heisenberg group}

We explicitly describe now the group $\hat{\Gamma}_{(0, -\infty)}$ in the same way, as the group   $\tilde{\Gamma}_{(0, -\infty)}$
was described in item~\ref{iso-Heis}  of Proposition~\ref{prop-Heis}.

We recall that  we have $m \diamond (n, -\infty) = (n, -\infty)$, where $m \in \dz = C$ and $n \in \dz = \pi(\Gamma)$. Therefore the group $\dz = C$
acts on the $\dz$-torsor $[n_1, n_2]$ for any $n_1, n_2 \in \dz= \pi(\Gamma)$. (We recall that we use notation $n_1 = (n_1, -\infty)$ and $n_2 = (n_2, -\infty)$).
We calculate this action in the following  lemma.

\begin{lemma}  \label{Calcul}
For any $m \in \dz=C$, for any $n_1, n_2 \in \pi(\Gamma)= \dz$ and any $h \in [n_1, n_2]$  we have
$$
m (h) = \frac{1}{2}m (n_1 +n_2 -1)(n_1 - n_2) + h  \, \mbox{.}
$$
\end{lemma}
\begin{proof}
We suppose that $n_1 \ge n_2$. (The case $n_1 < n_2$ follows from the case $n_1 \ge n_2$.)
We have that
$$
[n_1, n_2]= [n_1, n_1 -1] \otimes_{\dz} [n_1-1, n_1 -2] \otimes_{\dz}  \ldots \otimes_{\dz} [n_2 +1,n_2] \, \mbox{.}
$$
Therefore it is enough to calculate the action of the group $\dz =C$ on the $\dz$-torsor  $[k+1, k]$, where $k \in \dz= \pi(\Gamma)$. We have by construction of $[k+1, k]$ that
$$
m (s)= mk +s  \, \mbox{,} \quad \mbox{where}  \quad  m \in \dz=C \, \mbox{,} \quad  s \in [k+1, k]  \, \mbox{.}
$$
Hence and using the arithmetic progression,  we obtain
$$
m(h) = m(n_2 + (n_2 +1) + \ldots + (n_1 -1)) + h =  \frac{1}{2}m (n_1 +n_2 -1)(n_1 - n_2) + h \, \mbox{.}
$$
\end{proof}

\begin{Th} \label{quadruple}
Using the isomorphism which is inverse to isomorphism from item~\ref{iso-Heis}  of Proposition~\ref{prop-Heis} (after  fixing a section of the homomorphism $\pi$)
we present any element from the group $\hat{\Gamma}_{(0, -\infty)}$ as a quadruple $(a,b,c,m)$, where $(a,b,c) \in {\rm Heis}(3, \dz)$ and $m \in \dz$. Then, in this presentation, the group  $\hat{\Gamma}_{(0, -\infty)}$ is a set of integer quadruples with the group law written as:
$$
(a_1, b_1, c_1, m_1)(a_2, b_2. c_2, m_2) = (a_1 +a_2 + m_1 b_2, b_1 + b_2, c_1 +c _2 + a_1 b_2 + \frac{1}{2} b_2 (b_2 -1) m_1, m_1 + m_2 )  \, \mbox{.}
$$
\end{Th}
\begin{proof}
Using the isomorphism from item~\ref{iso-Heis}  of Proposition~\ref{prop-Heis} we calculate the action of the group $\dz=C$
on the group ${\rm Heis}(3, \dz)$ through the action on the group $\tilde{\Gamma}_{(0, -\infty)}$ described in section~\ref{act-ext-heis}.
For any $(a,b,c) \in {\rm Heis}(3, \dz)$ and for any $m \in \dz$ we have the action of $m$ on $(a,b,c)$:
\begin{multline}
m \left((a,b,c) \right)= m\left( (b \oplus a ), -c + d_b \right)= ( (b \oplus a + mb), -c +   m(d_b) )  \stackrel{\mbox{ \scriptsize Lemma}~\ref{Calcul}}{=} \\
=( (b \oplus a + mb), -c - \frac{1}{2}b (b-1) m + d_b )=(a + mb, b, c + \frac{1}{2}b (b-1) m   )   \, \mbox{,}   \label{m-act}
\end{multline}
where the element $d_b \in [0,b]$ was defined in  Proposition~\ref{prop-Heis}.

Now, using this calculation,  we have
\begin{multline*}
(a_1, b_1, c_1, m_1) (a_2, b_2, c_2, m_2) = ( (a_1, b_1, c_1 ) \  m_1 ((a_2, b_2, c_2)) , m_1 +m_2 ) = \\
=
((a_1, b_1, c_1 )  (a_2 + m_1 b_2, b_2, c_2 + \frac{1}{2}b_2 (b_2 -1) m_1 ), m_1 + m_2 )= \\=
(a_1 +a_2 + m_1 b_2,  b_1 + b_2,  c_1 +c _2 + a_1 b_2 + \frac{1}{2} b_2 (b_2 -1) m_1, m_1 + m_2   )   \, \mbox{.}
\end{multline*}
\end{proof}

\medskip

Let $G$ be a group of integer quadruples $(a,b,c,m)$ with the group law given in Theorem~\ref{quadruple}.

\medskip

\begin{nt} \label{section} \em
The isomorphism between the group $\hat{\Gamma}_{(0, -\infty)}$ and the group $G$ constructed in Theorem~\ref{quadruple}
depends on the choice of a section $s$ of the homomorphism $\pi$ from~\eqref{centr}. We denote this isomorphism by $\Upsilon_s$.
The group $C = \dz$ acts naturally on the set of sections of the homomorphism $\pi$.
Moreover, the set of all sections of the homomorphism $\pi$ is a  $C$-torsor. Now from construction of homomorphisms $\Upsilon_s$ and calculations
from formula~\eqref{m-act} we obtain that the following diagram is commutative for any section $s$ of the homomorphism $\pi$ and any $k \in \dz=C$:
 $$
 \xymatrix{
& \quad \hat{\Gamma}_{(0, -\infty)} \ar[dl]_{\Upsilon_s}  \ar[dr]^{\Upsilon_{-k+s}} & \\
G \ar[rr]_{\phi_{k}} && G
}
 $$
 where the automorphism $\phi_k$ is obtained from formula~\eqref{m-act}:
 $$
 \phi_k((a,b,c,m))= ( k((a,b,c)), m )= (a + kb, b, c + \frac{1}{2}b (b-1) k, m   )  \, \mbox{.}
 $$
\end{nt}

Using  Theorem~\ref{quadruple} we can describe properties of the extended discrete Heisenberg group.

\begin{prop}
We have the following properties of the group $G$.
\begin{enumerate}
\item  \label{iit1}
 The group G is isomorphic to a subgroup
of the group ${\rm GL}(4, \dz)$ via a map:
$$
(a,b,c,m)  \longmapsto
\begin{pmatrix}
1 & m & a & c \\
0 & 1 & b & \frac{1}{2}b(b-1) \\
0 & 0 & 1 & b \\
0 & 0 & 0 & 1
\end{pmatrix}
$$
\item  \label{iit2} The center $Z(G)$ of $G$ is isomorphic to $\dz \simeq \{ (0,0,c,0) \}$.
\item \label{iit3} The lower central series is:
$$
[G,G]= \{(a,0,c,0     )\} \simeq \dz \oplus \dz  \, \mbox{,}
\qquad \qquad
[[G,G], G] = {Z}(G)  \, \mbox{,}
$$
$$
\mbox{and}  \qquad
G / [G,G]  \simeq \dz \oplus \dz   \, \mbox{,} \qquad \qquad [G,G] / [[G,G],G ]  \simeq \dz  \, \mbox{.}
$$
 \end{enumerate}
\end{prop}
\begin{proof}
Items~\ref{iit1}-\ref{iit2} follow from direct computations with group laws.

For the proof of item~\ref{iit3} we use an explicit form of the lower central series for the group of upper triangular integer matrices with ones on the main diagonal: the $k$-th term of its lower central series is the set of all matrices from the group
such that the number of diagonals above the main diagonal that contain zeroes only is equal to $k-1$.
 Besides, we use the following identities:
$$
[(0,0,0,m), (0,b,0,0)]=(mb, 0, -\frac{1}{2}mb(b+1), 0 ) \, \mbox{,}
\qquad
[(a,0,c,0), (0,b,0,0)]= (0,0,ab, 0 )  \, \mbox{.}
$$
(The first identity follows easily from formula~\eqref{m-act}, the second identity is satisfied in the subgroup ${\rm Heis}(3, \dz)$.)

This gives the proof.
\end{proof}

\subsection{Extended discrete Heisenberg group and the Fourier transforms}

Let $\Gamma$ be a group as in the beginning of section~\ref{discr_2} (with the fixed central extension~\eqref{centr}).
We describe now the connection between Fourier transforms on the spaces of functions and distributions on the group $\Gamma$
and the action of the extended discrete Heisenberg group on these spaces.

We fix $r \in \dc^*$ and $\alpha \in \breve{\Gamma}$.

We recall that by Corollary~\ref{cons-me} of Proposition~\ref{inol-prop}   and by item~\ref{item-2} of Proposition~\ref{inol-prop}   we have  canonical isomorphisms of one-dimensional $\dc$-vector spaces for any $\gamma$ and   $\beta$ from  $\Gamma$:
$$
\mu_r(\alpha, \beta \circ \alpha) \simeq \mu_r(\alpha^{\perp(\gamma)}, (\beta \circ \alpha)^{\perp(\gamma)} )  \simeq
\mu_r(\alpha^{\perp(\gamma)},   (- \beta) \circ \alpha^{\perp(\gamma)} )   \, \mbox{.}
$$
We denote the composition of these isomorphisms as $\varepsilon$.

We need the following lemma.

\begin{lemma}  \label{lemma-last} Let $\gamma  \in  \pi^{-1}(0) \subset \Gamma$. Then we have properties.
\begin{enumerate}
\item  \label{i1}
For any $\beta \in \breve{\Gamma}$ and any $m \in \dz =C$ we have
 $(m \diamond \beta)^{\perp(\gamma)}= m \diamond \beta^{\perp(\gamma)}$  \, \mbox{.}
\item  \label{i2} For any $\alpha_1$ and $\alpha_2$ from $\breve{\Gamma}$ the following diagram is commutative
\begin{equation}  \label{diagram2}
 \xymatrix{
[\alpha_1, \alpha_2 ]  \ar[r]  \ar[d] & [\alpha^{\perp(\gamma)}, \beta^{\perp(\gamma)}] \ar[d]\\
[m \diamond  \alpha_1, m \diamond  \alpha_2 ]  \ar[r] & [  m \diamond \alpha_1^{\perp(\gamma)}, m \diamond \alpha_2^{\perp(\gamma)}]
}
 \end{equation}
where the horizontal arrows are isomorphisms which follow from item~\ref{item-3} of Proposition~\ref{inol-prop} and from item~\ref{i1} of this lemma,
and the vertical arrows are induced by the actions of the element $m \in \dz=C$.
 \item  \label{i3}
There is  a commutative diagram which is analogous to diagram~\ref{diagram2} when we change $\dz$-torsors $[ \cdot  , \cdot  ]  $ to corresponding one-dimensional $\dc$-vector spaces $\mu_r(\cdot, \cdot)$.
\end{enumerate}
\end{lemma}
\begin{proof}
We note that $m \diamond \gamma = \gamma$. Therefore  for any $\phi$ and $\psi$ from $\breve{\Gamma}$ we have that  $\phi \circ \psi \circ \gamma > 0$ if and only if
$$(m \diamond  \phi)  \circ (m \diamond \psi) \circ \gamma > 0  \mbox{.}$$
Hence we obtain item~\ref{i1}.

Item~\ref{i2} follows from the construction of isomorphism from item~\ref{item-3} of Proposition~\ref{inol-prop}.

Item~\ref{i3} follows from item~\ref{i2}.
\end{proof}

\begin{prop} \label{prop-last} There are isomorphisms.
\begin{enumerate}
\item \label{ite1}
Let $\gamma \in \Gamma$.
There is a canonical isomorphism of groups:
$$
\rho_{\alpha, \gamma} \; : \;   \tilde{\Gamma}_{\alpha}   \lrto \tilde{\Gamma}_{\alpha^{\perp(\gamma)}}
\qquad
\mbox{given as}
\qquad
(\beta, h)  \longmapsto (- \beta, \varepsilon( h))  \, \mbox{,}
$$
where $\beta \in \Gamma$, $h \in \mu_r(\alpha, \beta \circ \alpha)$,  and
$\varepsilon(h)  \in { \mu_r(\alpha^{\perp(\gamma)},   (- \beta) \circ \alpha^{\perp(\gamma)} ) }$.
\item \label{ite2} Let $\gamma  \in \pi^{-1}(0) \subset \Gamma$. There is a canonical isomorphism of groups:
$$
\varrho_{\alpha, \gamma} \; : \;   \hat{\Gamma}_{\alpha}   \lrto \hat{\Gamma}_{\alpha^{\perp(\gamma)}}
\qquad
\mbox{given as}
\qquad
(f, m)  \longmapsto (\rho_{\alpha, \gamma}(f), m)  \, \mbox{,}
$$
where $f \in \tilde{\Gamma}_{\alpha}$ and $m \in \dz$.
\end{enumerate}
\end{prop}
\begin{proof}
Item~\ref{ite1} follows from the definition of the group law and from Corollary~\ref{cons-me} of Proposition~\ref{inol-prop}.

For the proof of item~\ref{ite2} it is enough to prove that $\rho_{\alpha, \gamma} (m(f)) = m(\rho_{\alpha, \gamma}(f))$
for any $f \in \tilde{\Gamma}_{\alpha}$ and $m \in \dz$. This follows from item~\ref{i3} of Lemma~\ref{lemma-last}.
\end{proof}

We recall that the groups $\tilde{\Gamma}_{\alpha}$ and $\hat{\Gamma}_{\alpha}$ naturally act on the $\dc$-vector spaces
$\D_{+, \alpha, r}(\Gamma) $ and
$\D'_{+, \alpha, r}(\Gamma)$ (see section~\ref{action}  and formula~\eqref{actions-ex}). Using isomorphisms
$\rho_{\alpha, \gamma}$ (when $\gamma \in \Gamma$ is any element) and $\varrho_{\alpha, \gamma}$ (when $\gamma  \in  \pi^{-1}(0) \subset \Gamma$)
we obtain  natural actions  of the group  $\tilde{\Gamma}_{\alpha}$ (for any $\gamma \in \Gamma$) and the group  $\hat{\Gamma}_{\alpha}$ (for any $\gamma  \in  \pi^{-1}(0) \subset \Gamma$) on the $\dc$-vector spaces $\D_{+, \alpha^{\perp(\gamma)}, r}(\Gamma)$ and   $\D'_{+, \alpha^{\perp(\gamma)}, r}(\Gamma)$.

Now from the definition of Fourier transforms, from Corollary~\ref{cons-me} of Proposition~\ref{inol-prop}  and from Lemma~\ref{lemma-last}, we have a theorem.

\begin{Th} \label{th-last}
The Fourier transforms
$$
{\bf F}_{\gamma} \; : \; \D_{+, \alpha, r}(\Gamma)  \lrto \D_{+, \alpha^{\perp(\gamma)}, r}(\Gamma)     \qquad \mbox{and}  \qquad
{\bf F}_{\gamma} \; : \; \D'_{+, \alpha, r}(\Gamma)  \lrto \D'_{+, \alpha^{\perp(\gamma)}, r}(\Gamma)
$$
commute with the actions of the groups   $\tilde{\Gamma}_{\alpha}$  (when $\gamma \in \Gamma$ is any element)  and $\hat{\Gamma}_{\alpha}$
(when $\gamma  \in  \pi^{-1}(0) \subset \Gamma$).
\end{Th}

\begin{nt} \em
Item~\ref{ite2} of Proposition~\ref{prop-last}  and the part of Theorem~\ref{th-last} on the action of group $\hat{\Gamma}_{\alpha}$ is not true for arbitrary $\gamma \in \Gamma$. It is important that $m \diamond \gamma = \gamma$
for any $m \in \dz=C$.
\end{nt}

\vspace{0.3cm}

\noindent
D. V. Osipov

\noindent Steklov Mathematical Institute of Russsian Academy of Sciences, 8 Gubkina St., Moscow 119991, Russia, {\em and}

\noindent National Research University Higher School of Economics, Laboratory of Mirror Symmetry,  6 Usacheva str., Moscow 119048, Russia,
{\em and}

\noindent National University of Science and Technology ``MISiS'',  Leninsky Prospekt 4, Moscow  119049, Russia

\noindent {\it E-mail:}  ${d}_{-} osipov@mi{-}ras.ru$

\vspace{0.6cm}

\noindent
A. N. Parshin

\noindent Steklov Mathematical Institute of Russsian Academy of Sciences,

\noindent  8 Gubkina St., Moscow 119991, Russia

\noindent {\it E-mail:}  $parshin@mi{-}ras.ru$

\end{document}